\documentclass[11pt,twoside]{amsart}
\usepackage{amsfonts}
\usepackage{epsfig,graphics, graphicx}
\usepackage{amssymb}
\usepackage{tikz, tikz-cd}
\usepackage{hyperref, enumerate, indentfirst}
\hypersetup{
  linktocpage,
    colorlinks,
    citecolor=black,
    filecolor=black,
    linkcolor=blue,
    urlcolor=blue
    }

\newtheorem{theorem}{Theorem}[section]
\newtheorem{proposition}[theorem]{Proposition}
       
\newtheorem{lemma}[theorem]{Lemma}

\newtheorem{remark}[theorem]{Remark}
\newtheorem{corollary}[theorem]{Corollary}

\newtheorem{conjecture}{Conjecture}

\newcommand{\hm}{{\widehat{M}}}
\newcommand{\hx}{{\hat{x}}}
\newcommand{\hy}{{\hat{y}}}
\newcommand{\hz}{{\hat{z}}}

\newcommand{\diag}{\operatorname{diag}}
\newcommand{\ad}{\operatorname{ad}}

\newcommand{\OO}{\operatorname{O}}

\newcommand{\PO}{\operatorname{PO}}
\newcommand{\Ein}{\operatorname{\bf Ein}^{1,n-1}}

\newcommand{\Conf}{\operatorname{Conf}}

\newcommand{\Ad}{\operatorname{Ad}}

\newcommand{\Kill}{\operatorname{Kill}}
\newcommand{\Span}{\operatorname{Span}}

\newcommand{\BN}{{\bf N}} 
\newcommand{\BR}{{\bf R}}

\newcommand{\BH}{{\bf H}}

\newcommand{\lieu}{{\mathfrak{u}}}
\newcommand{\lieg}{{\mathfrak{g}}}

\newcommand{\lieo}{{\mathfrak{o}}}
\newcommand{\liep}{{\mathfrak{p}}}
\newcommand{\lies}{{\mathfrak{s}}}

\newcommand{\oo}{{\mathfrak{o}}}
\newcommand{\so}{{\mathfrak{so}}}
\newcommand{\liea}{{\mathfrak{a}}}
\newcommand{\liem}{{\mathfrak{m}}}
\newcommand{\lieq}{{\mathfrak{q}}}

\title[A local Lorentzian Ferrand-Obata theorem]{A local Lorentzian Ferrand-Obata theorem for conformal vector fields}
\author{Sorin Dumitrescu, Charles Frances, Karin Melnick, \\
Vincent Pecastaing, Abdelghani Zeghib}

\address{Universit\'e C\^ote d'Azur, CNRS, LJAD, France}

\email{Sorin.Dumitrescu@univ-cotedazur.fr}

\address{Institut de Recherche Math\'ematique Avanc\'ee, 7, Rue Ren\'e Descartes, Universit\'e de Strasbourg, 67085, Strasbourg, France}

\email{frances@math.unistra.fr}

\address{Department of Mathematics, 6, Avenue de la Fonte, University of  Luxembourg,  L-4364 Esch-sur-Alzette, Luxembourg}

\email{karin.melnick@uni.lu}

\address{Universit\'e C\^ote d'Azur, CNRS, LJAD, France}

\email{Vincent.Pecastaing@univ-cotedazur.fr}

\address{UMPA, CNRS, Ecole Normale Sup\'erieure de Lyon, France}

\email{abdelghani.zeghib@ens-lyon.fr}



\date{\today}
\thanks{We are grateful to the organizers of the IHP trimester Group Actions and Rigidity: Around the Zimmer Program, to the IHP itself (UAR 839 CNRS-Sorbonne Université), and to LabEx CARMIN (ANR-10-LABX-59-01), for providing the stimulating environment in which this work was initiated. This work has also been supported by the French government, through the UCA$^{\text{JEDI}}$ Investments in the Future project managed by the National Research Agency (ANR) with the reference number ANR-15-IDEX-01, by the ANR grant IsoMoDyn ANR-25-CE40-1360-03 and by the Luxembourg National Research
Fund (FNR), grant reference O24/18936913/RiGA.
We wish to expressly thank Thomas Leistner for helpful conversations and comments.}

\begin{document}
\setcounter{tocdepth}{1}

\begin{abstract}
    For a conformal vector field on a closed, real-analytic, Lorentzian manifold we prove that the flow is locally isometric---that it preserves a metric in the conformal class on a neighborhood of any point---or the metric is everywhere conformally flat. The main theorem can be viewed as a local version of the Lorentzian Lichnerowicz conjecture in the real-analytic setting. The key result is an optimal improvement of the local normal forms for conformal vector fields of \cite{fm.champsconfs}, which focused on non-linearizable singularities. 
   This article is primarily concerned with essential linearizable singularities, and the proofs include global arguments which rely on the compactness assumption. 
\end{abstract}

\maketitle

\section{Introduction}
\label{sec.intro}

The Ferrand-Obata Theorem on conformal groups of Riemannian manifolds is paradigmatic in the Zimmer-Gromov program, which aims to classify compact or finite-volume manifolds with rigid geometric structures admitting large group of automorphisms.This theorem was conjectured by Lichnerowicz and proved independently by J. Ferrand and M. Obata; it characterizes closed Riemannian manifolds with non-compact conformal transformation group (see Theorem \ref{Obata} below). In contrast, the Lorentzian version  of the  Lichnerowicz conjecture is still open, although significant progress has been made in certain cases.

In this article, we establish a local version of the Lorentzian Lichnerowicz conjecture in the real analytic category.  Our results shed light on the behavior around singularities of conformal vector fields and also constitute a maximality result with respect to conformal embeddings, having as a consequence that certain noncompact Lorentzian manifolds cannot be conformally compactified; thus they link local analysis with global geometric rigidity.
 
\subsection{The Lorentzian Lichnerowicz conjecture}
\label{sec.lichne}

For a semi-Riemannian metric $g$ on a manifold, a subgroup $G$ of its conformal group $\Conf(M,g)$ is called {\it essential} if it does not preserve any metric $g'$ in the conformal class of $[g]$, namely, of the form $g'=e^{2 \lambda}g$ for $\lambda$ a function on $M$ of suitable differentiability.
In the case of a Riemannian metric on a compact manifold, a conformal group $G$ is essential if and only if it is non-compact.  In this case J. Ferrand and M. Obata each proved independently :

\begin{theorem}[Ferrand \cite{lf.lich,lf.conf.regularity}, Obata \cite{obata.lich}] \label{Obata}
    Let $(M,g)$ be a compact Riemannian manifold of dimension $n \geq 2$. If $\Conf(M,g)$ is essential then $(M,g)$ is conformally diffeomorphic to the round sphere $\mathbf{S}^n$.
\end{theorem}

This theorem positively answers a question asked by A. Lichnerowicz. 

A similar question in the pseudo-Riemannian framework  was asked subsequently  in \cite[Section 6.2]{dag.rgs}. In higher signature, essentiality of the conformal group does not determine the topology of the manifold. In fact, in the Lorentzian case, the second author proved that, in any dimension $n \geq 3$, there are infinitely-many topological types of compact manifolds bearing infinitely-many non-conformally-equivalent Lorentzian metrics, each admitting an essential conformal flow \cite{frances.nogloballich}. Nevertheless, all these Lorentzian metrics are {\it conformally flat}, meaning they are locally conformally equivalent to flat Minkowski space.
The analogue of the original Lichnerowicz Conjecture is thus:

\begin{conjecture}[{\bf Lorentzian Lichnerowicz}]
    \label{conj.llc}
    Let $(M,g)$ be a compact Lorentzian manifold of dimension $n \geq 3$. If $\Conf(M,g)$ is essential then $(M,g)$ is conformally flat.
\end{conjecture}

Important results which support the Lorentzian Lichnerowicz Conjecture were obtained by the second and the third authors for real-analytic three-dimensional manifolds \cite{fm.lich3d} and by the third and the fourth authors for simply-connected real-analytic manifolds \cite{mp.dambraconf}. The fourth author proved the Lorentzian Lichnerowicz Conjecture in the case when $\Conf(M,g)$ contains an essential connected  simple Lie subgroup \cite{pecastaing.can.sl2r,pecastaing.ss}. This last result is a step  towards  the classification of conformal groups of compact Lorentzian manifolds, in the vein of the classification of their isometry groups obtained  in \cite{as.lorisom1,as.lorisom2,zeghib.lorisom1}.  The recent work \cite{mehidi} addresses the Lorentzian Lichnerowicz Conjecture for locally homogeneous spaces, and proves that compact quotients of conformally homogeneous Lorentzian spaces have essential conformal group if and only if they are conformally flat.

In higher signature the analogous pseudo-Riemannian Lichnerowicz Conjecture fails: the second author constructed non-conformally-flat pseudo-Riemannian metrics of signature $(p,q)$, with $p,q \geq 2$, admitting essential conformal flows in \cite{frances.pq.counterexs}. These pseudo-Riemannian spaces are obtained from non-conformally-flat  polynomial deformations $g$ of the $(p,q)$-Minkowski metric on 
${\BR}^{p+q}\setminus\{0 \}$, which conformally descend to compact quotients of the form $(\mathbf{S}^1 \times \mathbf{S}^{p+q-1},[g])$.

\subsection{Statement of the main theorem} 
\label{sec.statement}

Our main result can be framed in the context of Conjecture \ref{conj.llc} above.  We prove conformal flatness assuming a stronger notion of essentiality, namely that a one-parameter group of conformal transformations is essential.  This assumption implies essentiality of $\Conf(M,g)^0$; we are not sure whether it is stronger. 

A conformal vector field $Y$ on a semi-Riemannian manifold $(M,g)$ is {\it inessential} if there exists a metric in the conformal class $[g]$ for which $Y$ is a Killing field. A specific, but still challenging, case of Conjecture \ref{conj.llc} is the following.
 
\begin{conjecture}[{\bf Lorentzian Lichnerowicz for conformal vector fields}] 
\label{conj.llc.vf}
Let $(M,g)$ be a compact Lorentzian manifold of dimension $n \geq 3$. If $(M,g)$ is not conformally flat, then any conformal vector field on $M$ is inessential.
\end{conjecture}

In this statement, it is crucial to make the global assumption that $M$ is a compact manifold. D. Alekseevsky built a family of non-conformally-flat, real-analytic, Lorentzian metrics on $\BR^n$ admitting a linear conformal flow which is essential, in fact, essential on a neighborhood of every singularity \cite{aleks.selfsim}. 

Our main theorem is motivated by Conjecture \ref{conj.llc.vf}.  We assume real-analyticity and we obtain a local, rather than a global, conclusion.  A conformal vector field $Y$  is {\it locally inessential} if each point $p \in M$ admits a neighborhood $U$ on which $Y$ is a Killing field for some metric in the conformal class $[\left. g \right|_U]$. 

\begin{theorem} 
\label{thm.local.lichnerowicz}
Let $(M,g)$ be a closed, real-analytic Lorentzian manifold of dimension $\geq 3$. If $(M,g)$ is not conformally flat, then any conformal vector field on $M$ is locally inessential. 
\end{theorem}

Although the conclusion of the theorem as stated is local, the assumption that $M$ is a compact manifold is needed because of the examples in \cite{aleks.selfsim} mentioned above. Theorem \ref{thm.local.lichnerowicz}  proves, in particular, that non-conformally-flat  examples in \cite{aleks.selfsim} do not admit real-analytic conformal compactifications, that is, they do not conformally and equivariantly embed in a closed Lorentzian manifold.  In \cite{mp.confdambra} the third and fourth authors proved this under the additional assumption that $M$ is simply connected.

Given a conformal vector field $Y$ on a Lorentzian manifold $(M,g)$, and given a point $p \in M$, the \emph{conformal distortion} of $Y$ with respect to $g$ at $p$ is the real number $\lambda$ such that $(L_Yg)_p= 2\lambda g_p$. When $Y_p=0$, the distortion at $p$ is the same for every metric in the conformal class. This is a consequence of the formula 
$$L_Y(fg)=(Y.f)g+fL_Yg \qquad \forall f \in C^\infty(M)$$ 
We say $Y$ has \emph{nontrivial conformal distortion} at a singularity $p$ if $\lambda \neq 0$. Observe that a conformal vector field with nontrivial conformal distortion at a singular point $p$ cannot preserve any metric in the conformal class, even in a neighborhood of $p$.
 
A conformal vector field $Y$ is always inessential in a neighborhood of a nonsingular point $p$---that is, with $Y_p \neq 0$. Namely, for coordinates $(t, x^1, \ldots, x^{n-1})$ on a neighborhood $U$ of $p$ in which $Y$ is $\frac{\partial}{\partial t}$, let 
$$
h_{(t,x)}(u,v) = g_{(0,x)}\big(D_{(t,x)} \varphi^{-t}_Y(u), D_{(t,x)} \varphi^{-t}_Y(v)\big).
$$
Then $h \in [\left.g\right|_U]$ is invariant under the local flow of $Y$.  Thus the content of Theorem \ref{thm.local.lichnerowicz} is the inessentiality around singularities of $Y$, that is, around points where $Y_p = 0.$  At these points, the conclusion says that $Y$ is locally isometric for a metric in the conformal class.  Local Killing fields of a metric are linearizable via the exponential map of the metric.  A corollary of Theorem \ref{thm.local.lichnerowicz}, which can be compared to \cite[Thm 1.3]{fm.champsconfs}, is : 

\begin{corollary}
\label{coro.distortion}
Let $(M,g)$ be a closed, real-analytic Lorentzian manifold of dimension $\geq 3$, and suppose that $(M,g)$ is not conformally flat. Let $Y$ be a conformal vector field on $(M,g)$ and $p \in M$ a singularity of $Y$. Then $Y$ is analytically linearizable and has trivial conformal distortion at $p$. In particular, $Y$  is analytically conjugate in a neighborhood of $p$ to a Killing vector field of Minkwoski space.
\end{corollary}

\section{Reduction to linearizable singularities of mixed or balanced type}
\label{sec.outline}

As discussed in the previous section, a nonsingular point of the conformal vector field $Y$ is always locally inessential, so the focus will be on singularities for the remainder of the article.  
A singularity of $Y$ will be called \textit{locally essential} if there is a fundamental system of neighborhoods, on each of which $Y$ is essential, and the vector field $Y$ can be called \emph{locally essential} if every singularity of $Y$ is locally essential.

In Section \ref{subsec.linearizability} below, we recall the linearizability result of \cite{fm.champsconfs} and in Section \ref{ss.taxonomy} we categorize linearizable singularities of the conformal vector field $Y$ into several types.  For certain types, local inessentiality or conformal flatness of $(M,g)$ can be quickly deduced.  The remaining types are \emph{mixed} and \emph{balanced}, and the proofs for these two occupy Sections \ref{sec.mixed-flat}, \ref{sec.pp-wave}, and \ref{sec.balanced} below.

First, we construct an example illustrating that local inessentiality can coexist with global essentiality for a conformal vector field on a closed Lorentzian manifold. 

\subsection{Global versus local essentiality} 

It follows from Theorem \ref{Obata} that Riemannian essential conformal vector fields always have essential singularities. More precisely, non-compact flows by M\"obius transformations on ${\bf S}^n$ have one or two singularities, both essential.  The notions of essentiality and local essentiality coincide for Riemannian conformal vector fields.

In Lorentzian signature, essentiality and local essentiality are distinct notions. Consider the following example, the Lorentzian Hopf manifold. Endow $\BR^3 \setminus \{ 0\}$ with the metric $\tilde{g}_x :=  g_0/|x|^2$, where $g_0$ is the Minkowski metric $2dx_1dx_3 + dx_2^2$, and $|\cdot|$ denotes the Euclidean norm. Let $h$ be the homothetic transformation $x \mapsto 2x$, and let $\Gamma$ be the discrete group generated by $h$. The quotient of $\BR^3 \setminus \{ 0\}$ by $\Gamma$ is a smooth manifold $M$, diffeomorphic to $\mathbf{S}^1 \times \mathbf{S}^2$, endowed with the Lorentzian quotient metric, denoted $g$. The one-parameter group of homotheties of $\BR^3$ descends to an isometric $\mathbf{S}^1$-action on $(M,g)$; denote it by $\{k^t\}$. The linear $O(1,2)$-action descends to an essential conformal action on $(M,g)$. 

Consider a unipotent one-parameter subgroup $\{\tilde{u}^t\} < O(1,2)$. It induces a one-parameter subgroup $\{u^t\} < \Conf(M,g)$, generated by a conformal vector field $U$. The isotropic line of fixed points of $\{u^t \}$ projects onto two closed null geodesics $\Delta_i$, $i=1,2$, of $(M,g)$, which are  fixed pointwise by $\{u^t\}$. Given $p \in \Delta_i$, it is clear that the conformal distortion of $U$ at $p$ is null; moreover, the metric $g_0$ descends to a Lorentzian metric in the conformal class, well-defined in a neighborhood of $p$, which is $\{u^t\}$-invariant on this neighborhood.  In particular---see also Lemma \ref{lem.linearizable-homothetic} below---the flow $\{ u^t \}$ is locally inessential at $p$. 
However, $\{u^t\}$ is an essential conformal one-parameter subgroup of $\Conf(M,g)$. Indeed, all orbits of the flow spiral around either $\Delta_1$ or $\Delta_2$. If $g'$ were a $\{u^t \}$-invariant metric on $M$ in $[g]$, then for any $p \in M$, the values 
$$g'_p(U,U) = g'_{u^t(p)}(U,U) \rightarrow 0,$$ 
contradicting the fact that $U$ is spacelike on a dense, open subset. Thus $\{u^t\}$ is an essential conformal flow of $(M,g)$ all of whose singularities are locally inessential. 

The commutative product $\varphi^t = k^tu^t$ has no singularity on $M$, for all $t \neq 0$. It is another essential one-parameter group: indeed, if the convex $\{ k^t \}$-invariant subset $ \{g' \in [g] \ | \ (\varphi^t)^* g' = g'\}$ were non-empty, averaging over $\{k^t\}$ would yield a metric in the conformal class which is both $\{ k^t \}$-and $\{ \varphi^t \}$-invariant, therefore $\{ u^t \}$-invariant: a contradiction. Hence, even in dimension $3$, a Lorentzian conformal vector field can be everywhere non-singular and globally essential. 

\subsection{Linearizability of conformal vector fields}
\label{subsec.linearizability}

It is well known that, in the neighborhood of a singular point, any Killing field is linear in exponential coordinates. A vector field $Y$ is \emph{linearizable} on a neighborhood $U$ of a singularity $p$ if there are coordinates on $U$ in which the flow $\{ \varphi^t_Y \}$ equals the linear flow $\{ D_p \varphi^t_Y \}$, where defined.  Linearizability generally fails for conformal vector fields. Although some techniques are available to study such  fields near a singularity, there is currently no classification of their possible local normal forms in the smooth setting.

In the real-analytic case, the situation is simplified: non-linearizability occurs only in the conformally flat case, and then the list of possible normal forms are those of local conformal flows on Minkowski space, which can be calculated with Lie theory.  The following key result reduces the proof of Theorem~\ref{thm.local.lichnerowicz} to the case in which $Y$ is linearizable around every singularity.  

\begin{theorem} (\cite[Thm 1.2]{fm.champsconfs})
\label{thm.linearizable}
Let $(M,g)$ be a connected, real-analytic Lorentzian manifold of dimension $\geq 3$. Let $Y$ be a local conformal vector field vanishing at $p \in M$. If $(M,g)$ is not conformally flat, then $Y$ is linearizable in a neighborhood of $p$.
\end{theorem}

Linearizable conformal vector fields are locally homothetic.  It is likely that this fact already appears in the literature, but we do not know a reference, thus we provide a short proof. 

\begin{lemma}
\label{lem.linearizable-homothetic}
Let $(M,g)$ be a semi-Riemannian manifold, and $Y$ a conformal vector field vanishing at $p \in M$.  Assume that 
$Y$ is linearizable in a neighborhood $U$ of $p$, and that the  conformal distortion of $Y$ at $p$ 
is $a$. Then there exists $h \in [\left. g\right|_U]$ for which $Y$ is homothetic with distortion $a$, that is, $L_Y h = 2 a h$.
\end{lemma}

\begin{proof}
Let $a$ be the conformal distortion of $Y$ at the singularity $p$.
Since $Y$ is a conformal vector field, there is a smooth function $\sigma \in C^\infty(M)$ such that $L_Y g = 2 \sigma g$. The metric $g$ defines a volume form $\omega_g$, for which $L_Y \omega_g = n \sigma \omega_g$. In the local linearization around $p$, the flow is equivalent to its derivative, which is a linear homothety.  Pulling back the volume on $\BR^n$ by this coordinate chart gives a volume form $\omega_0$ on a neighborhood $U$ for which $L_Y \omega_0 = n a \omega_0$. After replacing $\omega_0$ with $-\omega_0$ if necessary, we may write $\omega_g = e^{n \lambda} \omega_0$ on $U$, for some $\lambda \in C^\infty(U)$. It follows that
$$
L_Y \omega_g = (n (Y.\lambda) + n a) e^{n \lambda} \omega_0 = (n (Y.\lambda) + n a) \omega_g.
$$
Together with $L_Y \omega_g = n \sigma \omega_g$, this yields
$$
\sigma - Y.\lambda = a.
$$
Let $h := e^{-2 \lambda} g$.  
On the open set $U$,
$$
L_Y h = L_Y(e^{-2 \lambda} g) = - 2 Y.\lambda e^{-2 \lambda} g + e^{-2 \lambda} L_Y g = 2 e^{-2 \lambda} (- Y.\lambda + \sigma)  g = 2 a h.
$$
In other words, $Y$ is a homothetic vector field for $h$ with conformal distortion at $p$ equal $a$. 
\end{proof}

Under the hypotheses of Theorem~\ref{thm.local.lichnerowicz}, and because of Theorem~\ref{thm.linearizable}, Lemma~\ref{lem.linearizable-homothetic} implies, in particular, that around a singularity with trivial conformal distortion, a conformal vector field is locally inessential.

\subsection{Taxonomy of singularities for a linearizable conformal vector field}
\label{ss.taxonomy}

Let $Y$ be a local conformal vector field on a Lorentzian manifold $(M,g)$. Assume that $Y_p=0$, and that $Y$ is linearizable in a neighborhood of $p$.  
In what follows, the local flow generated by $Y$ is denoted by $\{\varphi^t_Y\}$ (as $Y$ is not assumed complete, $\{\varphi^t_Y\}$ may be only a local flow, defined on a neighborhood of $(0,p) \in \BR \times M$). The derivative flow $\{D_p \varphi^t_Y\}$ can be written as a product $\{e^{a t} A^t\}$, where $a \in \BR$ is the conformal distortion of $Y$ at $p$, and $\{A^t\}$ is a one-parameter subgroup of $\operatorname{O}^0(1,n-1)$. Such one-parameter subgroups fall into three categories (see e.g. \cite{Ratcliffe}, \cite{Matsumoto}).  Recall that $\mbox{O}^0(1,n-1) \cong \mbox{Isom}^0({\bf H}^{n-1})$.

\begin{itemize}
\item {\it elliptic flows.} These are conjugate to one-parameter subgroups of the compact group $\operatorname{SO}(n-1) \subset \operatorname{O}^0(1,n-1)$.  
\item {\it parabolic flows.} These have exactly one fixed point on the boundary $\partial \BH^{n-1}$ and can be written as a commuting product $\{K^t U^t\}$ of an elliptic and a nontrivial unipotent one-parameter subgroup.  
\item {\it loxodromic flows.} These have exactly two fixed points on $\partial \BH^{n-1}$, and can be written as a commuting product $\{K^t D^t\}$ of an elliptic and a nontrivial hyperbolic one-parameter subgroup. 
\end{itemize}

For linear flows $\{D_p \varphi^t_Y\} = \{e^{a t} A^t\}$ with $A^t$ loxodromic, we may write in a suitable frame of $T_pM$:
\begin{equation}
    \label{eqn.loxo.diffl}
D_{p} \varphi^t_Y =
\begin{pmatrix}
e^{(a+b)t} & & \\
& e^{a t} R^t & \\
& & e^{(a-b)t}
\end{pmatrix},
\end{equation}
with $b  > 0$, and $\{R^t\} < \operatorname{SO}(n-2)$ a one-parameter subgroup.
The linearizable singularities of $Y$ can be classified as follows.

\begin{enumerate}
\item \emph{inessential}: 
 the conformal distortion of $Y$ at $p$ vanishes. 
\item \emph{contracting or expanding}:
 $\lim_{t \to \pm\infty} |D_p \varphi^t_Y| = 0$ 
\item \emph{mixed}: 
$\{D_p\varphi^t_Y\}$ has the form (\ref{eqn.loxo.diffl}) with $|b| > |a| > 0$  
\item \emph{balanced}: 
$\{D_p\varphi^t_Y\}$ has the form (\ref{eqn.loxo.diffl})
with $|b| = |a| \neq 0$
\end{enumerate}

 If $p$ is an expanding singularity of $Y$, then it is a contracting singularity of $-Y$.  Note that contracting singularities need not have $\{ A^t \}$ loxodromic. For instance, if $\{D_p \varphi^t_Y\}$ is of the form $\{ e^{a t} K^t U^t\}$ with $a < 0$, and $K^t$ and $U^t$ are commuting elliptic and parabolic flows of $\operatorname{O}^0(1,n-1)$, then $p$ is a contracting singularity (and it is expanding if $a > 0)$.  In the loxodromic case, when $\{D_p\varphi^t_Y\}$ has the form (\ref{eqn.loxo.diffl}), then contracting corresponds to $a<-b<0$. 
 In any case, existence of a contracting singularity is known to imply conformal flatness, by the following theorem.

\begin{theorem}[\cite{frances.degenerescence}, Thm~1.3~(2)]
\label{thm.contraction}
Let $(M,g)$ be a smooth Lorentzian manifold. Assume that $(f_k) \subset \operatorname{Conf}^{loc}(M,g)$, is a sequence of local conformal transformations, all defined on an open set $U$.  If, for some point $x_0 \in M$, the images $f_k(U) \to x_0$ in the Hausdorff topology, then $U$ is conformally flat. In particular, if $(M,g)$ is real-analytic, it is conformally flat.
\end{theorem}

When $p$ is a contracting singularity and $(M,g)$ is real-analytic, it is evident from the linearization that the local flow $\{ \varphi^t_Y\}$ is defined for all $t \geq 0$ on a neighborhood $U$ of $p$ and converges uniformly on compact subsets of $U$ to the constant map $p$ as $t \rightarrow + \infty.$  In the smooth setting, the same behavior is ensured by the Hartman-Grobman Theorem \cite[Theorem 6.3.1]{katok.hasselblatt}. It now follows:
\begin{corollary}
\label{coro.contraction}
Let $(M,g)$ be a smooth Lorentzian manifold. Suppose that $M$ admits a local conformal vector field with a contracting or expanding singularity $p$. Then a neighborhood of $p$ is conformally flat. In particular if $(M,g)$ is real analytic, it is conformally flat.
\end{corollary}

Thus under the hypotheses of Theorem \ref{thm.local.lichnerowicz}, expanding or contracting singularities do not occur.

\subsection{Summary}
\label{ss.mixed-balanced}

The discussion above shows that, given a real-analytic Lorentzian manifold $(M,g)$ which is not conformally flat, and a conformal vector field $Y$ on $M$, the field $Y$ is locally inessential or it has a singularity of mixed or balanced type, as defined in (3) and (4) above.

In the remainder of the paper, we shall show that under the additional assumption that $M$ is compact, such singularities imply $(M,g)$ is conformally flat. The proofs in these cases are more subtle, as they rely on global dynamical arguments related to the compactness of $M$.  

The case of a mixed-type singularity will be treated in Section~\ref{sec.mixed-flat}. Proposition \ref{prop.mixed} says that the presence of such a singularity on a compact real-analytic Lorentzian manifold implies the existence of a contracting singularity elsewhere in the manifold. Conformal flatness will follow by Corollary~\ref{coro.contraction}.

The case of a balanced-type singularity is the most delicate and will occupy a large part of the article. It was proved in \cite[Prop 4.1]{mp.confdambra} that existence of a balanced-type singularity implies conformal flatness or existence of a real-analytic codimension-one foliation. As in that proof, we show in Theorem \ref{thm.pp-wave} that in a neighborhood of a balanced-type singularity, there exists a \emph{gravitational pp-wave} metric in the conformal class. Combined with the assumption of real-analyticity, this leads to what is called a \emph{polarization} of the conformal structure; see the beginning of Section \ref{sec.pp-wave} below for definitions. In Section~\ref{sec.balanced}, global dynamical arguments will then show that a second polarization must appear, which ultimately implies conformal flatness and completes the proof of Theorem~\ref{thm.local.lichnerowicz}.

\section{The Cartan geometry framework and Gromov's Stratification}
\label{sec.geometric}

\subsection{Conformal Cartan connections and curvature}
\label{sec.cartan}

The {\it normal Cartan connection} associated with a conformal Lorentzian structure will play an important role in our proofs. We recall here some definitions and basic results about this framework.

\subsubsection{Equivalence problem for Lorentzian conformal structures}
\label{sec.equivalence}

Let $\BR^{2,n}$ denote $\BR^{n+2}$ with standard basis $\{ e_0, \ldots, e_{n+1} \}$, equipped with the quadratic form
$$Q_{2,n}(x) = 2x_0 x_{n+1} + 2x_1 x_n + x_2^2 + \cdots + x_{n-1}^2.$$
The Lorentzian Einstein Universe, denoted $\Ein$, is the projectivization of the null cone 
$$\mathcal{N}^{2,n} \setminus \{0\} = \{ {\bf x} \in \BR^{n+2} \setminus \{0\} \mid Q_{2,n}({\bf x}) = 0\}.$$

It is a smooth quadric hypersurface of ${\bf RP}^{n+1}$ that naturally inherits a conformal class $[g_{1,n-1}]$ of Lorentzian signature from the ambient quadratic form on $\BR^{2,n}$. It is diffeomorphic to ${\bf S}^1 \times_\iota {\bf S}^{n-1}$, where $\iota$ is inversion on both factors. By construction, there is a transitive conformal action of the group $G = \PO(2,n)$ on $\Ein$, and in fact $\Conf(\Ein,[g_{1,n-1}]) \cong G$. 

Now $\Ein$ is a compact, conformally homogeneous space, identified with $G/P$, where $P$ is the parabolic subgroup of $G$ stabilizing an isotropic line in $\BR^{2,n}$. It was proved by \'E. Cartan that in dimension $n \geq 3$, each conformal Lorentzian structure defines a unique Cartan geometry infinitesimally modeled on $\Ein$ in the following sense. 

\begin{theorem}[\'E. Cartan; see \cite{sharpe} Ch.V, and \cite{cap.slovak.book.vol1}, Sec.~1.6]
\label{thm.cartan}
Let $M$ be a connected manifold of dimension $n \geq 3$. A conformal Lorentzian structure $[g]$ on $M$ canonically determines:
\begin{itemize}
\item a principal $P$-bundle $\pi: \hm \rightarrow M$; and 
\item a regular, normal Cartan connection: a one-form $\omega : \hm \to \lieg$ satisfying, for all $\hat{x} \in \hm$,
\begin{enumerate}
 \item $\omega_{\hat{x}} : T_{\hat{x}} \hm \stackrel{\sim}{\rightarrow} \lieg$ is a linear isomorphism;
 \item $(R_p)^*\omega = \mathrm{Ad}(p^{-1}) \circ \omega$ for all $p \in P$;
 \item $\omega \!\left( \frac{d}{dt} (\hat{x} \cdot e^{tY}) \right) \equiv Y$ for all $Y \in \liep$.
\end{enumerate}
\end{itemize}
\end{theorem}

Regularity and normality are technical conditions on the curvature of $\omega$ (see Section \ref{sec.curvature} below) ensuring uniqueness. For their detailed definitions, refer to \cite{sharpe} and \cite{cap.slovak.book.vol1}. 

For the homogeneous model $\Ein$, the Cartan bundle is the group $G$, itself, which naturally fibers over $G/P$. 
The regular, normal connection associated with the conformal structure is in this case the Maurer--Cartan connection $\omega^G$.

Given the pair $(\hm, \omega)$, the conformal structure $[g]$ is recovered as follows. Let $x \in M$ and $\hx \in \hm_x$. For any $u \in T_xM$ and any $\hat{u} \in T_{\hx}\hm$ such that $\pi_*(\hat{u}) = u$, define $\iota_{\hx}(u)$ to be the projection of $\omega_{\hx}(\hat{u})$ to $\lieg/\liep$. This yields a well-defined linear isomorphism $\iota_{\hx}: T_xM \to \lieg/\liep$.  
Equivariance of $\omega$ with respect to the right-$P$-action on $\hm$, condition (2) in Theorem \ref{thm.cartan}, implies that for every $\hx \in \hm$ and every $p \in P$:
\begin{equation}
\label{eq.equivariant-iota}
 \iota_{\hx \cdot p} = \Ad(p^{-1}) \circ \iota_{\hx}.
\end{equation}
 The representation of $P$ on $\lieg/\liep$ derived from the adjoint representation factors through the quotient $P/G_+$, where $G_+$ is the unipotent radical of $P$, and preserves a unique similarity class, call it $[\ \mathbb{I} \ ]$, of Lorentzian scalar products on $\lieg/\liep$. For every $x \in M$ and $\hx$ in the fiber over $x$:
\begin{equation}
\label{eq.conformal-class}
[g_x]= [ \ \iota_{\hx}^* \mathbb{I} \ ].
\end{equation}

\subsubsection{Conformal curvature}
\label{sec.curvature}

Let $[g]$ be a Lorentzian conformal structure on a manifold $M$, and let $\omega$ denote the normal Cartan connection associated with $[g]$ as in Theorem \ref{thm.cartan}. The curvature of $\omega$, denoted $K$, is the $\lieg$-valued two-form on $\hm$ defined by
\begin{equation}
\label{eq.curvature}
 K(X,Y) = d\omega(X,Y) + [\omega(X), \omega(Y)],
\end{equation}
for every pair $X,Y$ of vector fields on $\hm$. Property (3) of $\omega$ implies that $K$ is semibasic. Moreover, the normalization condition on $\omega$ implies it is \emph{torsion-free}, that is, the two-form $K$ takes values in $\liep$. 

Using the parallelization $T\widehat{M} \cong \widehat{M} \times \lieg$ provided by $\omega$, we can encode $K$ as a function
$$ \kappa : \widehat{M} \rightarrow \wedge^2 (\lieg/\liep)^* \otimes \liep.$$
With respect to the $P$-representation on
$\wedge^2 (\lieg/\liep)^* \otimes \liep$ by 
$$(p.\gamma)(u,v) = \Ad(p) \gamma(\Ad(p^{-1})u, \Ad(p^{-1})v) \qquad \forall u,v \in \lieg/\liep$$ 
the curvature function $\kappa$ satisfies the equivariance property
$$ \kappa_{\hx \cdot p} = p^{-1}.\kappa_{\hx} \qquad \forall \hx \in \hm, \ p \in P$$

Vanishing of the Cartan curvature over an open subset $U \subset M$, that is, on $\pi^{-1}(U)$, is equivalent to conformal flatness of $U$. 
The Weyl $(3,1)$-tensor $W$ associated to the conformal class $[g]$ 
corresponds via $\iota_{\hat{x}}$ to composing $\kappa$ with the quotient epimorphism $\liep \rightarrow \liep / \lieg_+$:
$$ W_x(u,v,w) = \overline{\kappa}_{\hx}(\iota_{\hx}(u), \iota_{\hx}(v)) .\iota_{\hx}(w) \qquad \forall x \in M, \ u,v,w \in T_xM$$
Thanks to the equivariance properties of $\iota$ and $\kappa$, this expression is independent of the choice of $\hx$ in the fiber over $x$.

\subsubsection{Lifting conformal transformations and conformal vector fields}
\label{sec.lift}

Conformal transformations of $M$ lift naturally and uniquely to bundle automorphisms of $\widehat{M}$ leaving $\omega$ invariant; conformal vector fields similarly lift to unique $P$-invariant vector fields on $\widehat{M}$ whose Lie derivative annihilates $\omega$. Conversely, any bundle automorphism of $\hm$ preserving $\omega$ induces a conformal diffeomorphism of $M$, and similarly any right-$P$-invariant vector field on $\hm$ preserving $\omega$ induces a conformal vector field on $M$. The lifted action of $\Conf(M,[g])$ preserves the parallelization of $\widehat{M}$ determined by $\omega$, as in property (1) of $\omega$ in Theorem \ref{thm.cartan}, and is therefore free (see \cite[I.3.2]{kobayashi.transf}). Similarly, lifts of conformal vector fields to $\widehat{M}$ are nonvanishing, 
and are entirely determined by their value at a point.

Let $X$ be a conformal vector field on $M$, and lift it to $\hm$ as explained above. For any vector field $Z$ on $\hm$, 
$$ L_X \omega(Z) = X.\omega(Z) - \omega([X,Z]) = 0.$$
Together with the definition of the Cartan curvature
$$ K(X,Z) = d \omega(X,Z) + [\omega(X),\omega(Z)] = X.\omega(Z) - Z.\omega(X) - \omega([X,Z]) + [\omega(X),\omega(Z)],$$
we obtain
\begin{equation}
 \label{eq.killing.cartan}
 Z.\omega(X) = [\omega(X),\omega(Z)] - K(X,Z).
\end{equation}

For $\hat{\alpha}$ any smooth path in $\hm$, set $\xi(t) = \omega(X(\hat{\alpha}(t)))$.  Then equation \eqref{eq.killing.cartan} leads to the following first-order linear ODE for $\xi$:
\begin{equation}
 \label{eq.killing.transport}
 \xi'(t) = [\xi(t), \omega(\hat{\alpha}'(t))] - \kappa(\xi(t), \omega(\hat{\alpha}'(t))).
\end{equation}

This is the {\it Killing transport equation} along $\hat{\alpha}$.

\subsubsection{Conformal exponential map}
\label{sec.exponential}

For $Z \in \lieg$ let $\hat Z \in \mathcal{X}(\hm)$ be defined by $\omega(\hat Z) \equiv Z$. Let $\{\varphi^t_{\hat Z}\}$ denote the local flow on $\hm$ generated by $\hat Z$. At each $\hx \in \hm$, define ${\mathcal W}_\hx \subset \lieg$ as the set of $Z \in \lieg$ such that $\varphi_Z^t$ is defined for $t \in [0,1]$ at $\hx$. The \emph{exponential map at $\hx$} is
$$ \exp(\hx, \cdot) : {\mathcal W}_\hx \to \hm, \qquad \exp(\hx, Z) = \varphi_{\hat Z}^1 \cdot \hx.$$

As usual, the domain ${\mathcal W}_\hx$ is a neighborhood of $0$ for all $\hat{x}$, and the map $\xi \mapsto \exp(\hx, \xi)$ determines a diffeomorphism from a neighborhood ${\mathcal V}_\hx \subset {\mathcal W}_\hx$ of $0$ onto a neighborhood of $\hx$ in $\hm$. Moreover, for $\lieg_{-} \subset \lieg$ any subspace transverse to $\liep$, the map $\pi \circ \exp_\hx$ induces a diffeomorphism from a neighborhood of $0$ in $\lieg_{-}$ to a neighborhood of $x = \pi(\hx)$ in $M$.

Let $f$ be a conformal transformation of $M$. Then $f_*(\hat Z) = \hat Z$, and if $p \in P$, then $(R_p)_*\hat Z = \widehat{Z \cdot p}$, where $Z \cdot p = (\Ad p^{-1}) Z$. This implies the important equivariance property:
\begin{equation}
\label{eq.equivariance-exponential}
 f(\exp(\hx,\xi)) \cdot p^{-1} = \exp(f(\hx) \cdot p^{-1}, (\Ad p)\xi).
\end{equation}

\subsubsection{Development of curves and null geodesic segments}
\label{sec.lightlike-geodesic}
Recall that a \emph{pregeodesic} of a Lorentzian metric $g$ is a parametrized curve $\gamma: I \to M$ satisfying a differential equation of the form
$$
\frac{D}{dt} \dot{\gamma}(t) = f(t) \dot{\gamma}(t),
$$
for some smooth function $f: I \to \BR$.  
After a suitable reparametrization, a pregeodesic which is an immersion becomes an affinely parametrized geodesic.

A remarkable fact of pseudo-Riemannian conformal geometry is that all metrics in the same conformal class $[g]$ have the same \emph{null pregeodesics}.  A null pregeodesic is a pregeodesic $\gamma$ of any metric $g_0 \in [g]$ with $g_0(\dot{\gamma}(t), \dot{\gamma}(t)) \equiv 0$, the latter property being independent of the choice of $g_0$.  For $\gamma$ a null pregeodesic defined on $I$, the image $\gamma(I)$ will be called a \emph{null geodesic segment}.  

The Cartan connection defines a \emph{developing map}, which associates to any curve in $M$ a curve in the model space $\Ein$.  
Let $\gamma: I \to M$ be a smooth curve, and  consider a lift $\hat{\gamma}: I \to \widehat{M}$.  
Denote by $\hat{\gamma}_*$ the unique curve in $G = \PO(2,n)$ satisfying the differential equation
$$
\omega^G(\hat{\gamma}_*'(t)) = \omega(\hat{\gamma}'(t)),
$$
with initial condition $\hat{\gamma}_*(0) = 1_G$, where, as above $\omega^G$ is the Maurer-Cartan form on $G$, the Cartan connection for $\Ein$. The ODE defining  $\hat{\gamma}_*$ is linear, so the curve $\hat{\gamma}$ is defined on $I$.
Set $\gamma_*(t) := \pi_G(\hat{\gamma}_*(t))$, and call $\gamma_*$ the \emph{development} of $\gamma$.  
If $\hat{\beta}: I \to \widehat{M}$ is another lift of the same curve $\gamma$, then there exists $p \in P$ such that $\beta_* = p \cdot \gamma_*$; this is easy to check from the axioms for $\omega$.  
Thus any $P$-invariant family $\mathcal{F}$ of curves on $\Ein$ defines a distinguished class of curves on $M$, namely, those curves whose developments are in $\mathcal{F}$.

A \emph{photon} of $\Ein$ is defined as the projection onto $\Ein$ of a totally isotropic two-plane in $\BR^{2,n}$.  
The null pregeodesics are precisely those curves whose developments parametrize segments of photons (see \cite{francoeur}, Thm.~5.3.3).  
The conformal invariance of null pregeodesics asserted above can now easily be deduced. 
The following consequence of this characterization of null pregeodesics will be used repeatedly below: for any 
$\hx \in \hm$, for any $u \in \lieg$
with projection $\bar{u}$ to $\lieg/\liep$ satisfying $\mathbb{I}( \bar{u},\bar{u} )=0$, the parametrization 
$s \mapsto \pi \circ \exp(\hx,su)$ is an immersion onto a null geodesic segment in $M$.

\subsection{Algebraic features of $\mathfrak{so}(2,n)$}
\label{sec.algebraic-facts}

 Let $\mathbb{I}_{2,n}$ denote the inner product determined by the quadratic form $Q_{2,n}$, and let $\mathbb{I}_{1,n-1}$, or simply $\mathbb{I}$, be its restriction to the Minkowski subspace $e_0^\perp \cap e_{n+1}^\perp$.

The Lie algebra $\lieg = \so(2,n)$ can be parametrized as follows:
\begin{equation*}
\lieg= 
\left \{
\begin{pmatrix}
a & \xi & 0 \\
v & X & - \mathbb{I}^t \xi \\
0 & -^t v \mathbb{I} & -a
\end{pmatrix}
, \ a \in \BR, \ v \in \BR^n, \ \xi \in \BR^n, \ X \in \so(\BR^n, \mathbb{I}) \cong \so(1,n-1)
\right \}.
\end{equation*}

The \emph{grading decomposition} $\lieg = \lieg_{-1} \oplus \lieg_0 \oplus \lieg_{+1}$ is given by the subalgebras parametrized by $v$, $(a,X)$, and $\xi$, respectively; note that $\liep = \lieg_0 \ltimes \lieg_{+1}$. The Lie algebra $\lieg_0$ is reductive, isomorphic to $\BR \oplus \so(1,n-1).$
In turn $\so(1,n-1)$ can be decomposed as:
\begin{equation*}
\so(1,n-1) = 
\left \{
\begin{pmatrix}
b & U_+ & 0 \\
U_- & R & -^t U_+ \\
0 & -^t U_- & -b
\end{pmatrix}
, \ b \in \BR, \ U_- \in \BR^{n-2}, \ U_+ \in (\BR^{n-2})^*, \ R \in \so(n-2)
\right \}.
\end{equation*}

An $\BR$-split Cartan subalgebra in $\lieg$ is
\begin{equation}
\label{eqn:cartan_subalgebra}
\liea =
\left \{
\begin{pmatrix}
a & & & & \\
  & b & & & \\
  & & {\bf 0} & & \\
  & & & -b & \\
  & & & & -a
\end{pmatrix}
, \ a,b \in \BR
\right \}, 
\qquad {\bf 0} = (0, \ldots, 0) \ n-2 \ \text{times}.
\end{equation}

Let $\alpha, \beta \in \liea^*$ be defined by $\alpha(a,b)=a$ and $\beta(a,b)=b$. The following diagram represents the full restricted root-space decomposition of $\so(2,n)$ with respect to $\liea$:
\begin{equation}
\label{eqn:root_decomposition}
\begin{pmatrix}
\liea & \lieg_{\alpha - \beta} & \lieg_{\alpha} & \lieg_{\alpha+\beta} & 0 \\
\lieg_{\beta - \alpha} & \liea & \lieg_{\beta} & 0 & \lieg_{\alpha+\beta} \\
\lieg_{-\alpha} & \lieg_{-\beta} & \liem & \lieg_{\beta} & \lieg_{\alpha} \\
\lieg_{-\alpha-\beta} & 0 & \lieg_{-\beta} & \liea & \lieg_{\alpha-\beta} \\
0 & \lieg_{-\alpha-\beta} & \lieg_{-\alpha} & \lieg_{\beta-\alpha} & \liea
\end{pmatrix}
\qquad
\begin{array}{l}
\liem \cong \so(n-2) \\
\dim \lieg_{\beta} = n-2 = \dim \lieg_{\alpha} \\
\dim \lieg_{\alpha-\beta} = 1 = \dim \lieg_{\alpha+\beta}
\end{array}
\end{equation}

The factor $\liem$ is the centralizer of $\liea$ in a maximal compact subalgebra of $\so(2,n)$ and is parametrized by $R$ in the decomposition of $\so(1,n-1)$. In accordance with this decomposition, we denote the root spaces $\lieg_{-\beta}$ and $\lieg_\beta$ by $\lieu^-$ and $\lieu^+$, respectively. The corresponding unipotent subgroups of $G$ are $U^-$ and $U^+$.

For later use, we fix a basis $E_1, \ldots, E_n$ of $\lieg_{-1}$ such that $E_1 \in \lieg_{-\alpha+\beta}$, $E_n \in \lieg_{-\alpha-\beta}$, and $E_i \in \lieg_{-\alpha}$ for $2 \le i \le n-1$. We also impose the normalization $\langle E_1, E_n \rangle = 1$ and $\langle E_i, E_j \rangle = \delta_{ij}$ for $2 \le i,j \le n-1$.

The subgroup of $P$ with Lie algebra $\lieg_0$ will be denoted $G_0$.

\subsection{Gromov stratification for real-analytic structures and algebraicity of isotropy}
\label{sec.stratification}

A conformal class $[g]$ in dimension $\geq 3$ is a rigid geometric structure in the sense of Gromov \cite{gromov.rgs}. In analytic regularity, the orbits of local conformal vector fields are the fibers of a certain surjective analytic map onto a stratified analytic space. The same is true if instead of the orbits of all local conformal vector fields, one considers only those that centralize a given global analytic vector field $Y$. This amounts to saying that the enhanced geometric structure $[g] \cup \{Y\}$ is rigid and analytic and therefore the stratification result applies to it. We will denote by $\Kill^{loc}_Y(x)$ the Lie algebra of local conformal vector fields defined in a neighborhood of $x$ that centralize $Y$. In the context of pseudo-Riemannian geometry, the important consequences of Gromov's Stratification Theorem are summarized here.

\begin{theorem}[\cite{gromov.rgs}, §3]
\label{thm.stratification}
    Let $(M,[g])$ be a closed manifold of dimension $\geq 3$ with an analytic semi-Riemannian conformal structure. 
    Let $Y \in \mathcal{X}(M)$ be an analytic vector field. Then
    \begin{enumerate}
        \item For all $x \in M$, the $\Kill^{loc}_Y$-orbit $\mathcal{O}(x)$ is a locally closed, semi-analytic submanifold of $M$.
        \item For every $x$, the closure $\overline{\mathcal{O}(x)}$ is semi-analytic and locally connected.
        \item For all $y \in \overline{\mathcal{O}(x)}  \setminus \mathcal{O}(x)$, $\dim \mathcal{O}(y) < \dim \mathcal{O}(x)$.
    \end{enumerate}
\end{theorem}

\begin{remark}
These statements can also  be deduced from several analogous  results formulated in the framework of Cartan geometries. For orbits of all local conformal vector fields, the stratification result is proved in \cite[Sec.4]{me.frobenius}. A Frobenius theorem, the key part in this theory, is also proved in \cite{pecastaing.frobenius} for enhanced Cartan geometries, which include geometric structures such as $[g] \cup \{Y\}$. The setting is in $\mathcal{C}^{\infty}$ regularity and the conclusions are valid over an open-dense subset. However, the method of \cite{me.frobenius} are easily adaptable to this context, and yield a Frobenius theorem   for compact, real-analytic, enhanced Cartan geometries, valid on the entire manifold, which is exactly what is needed here.
\end{remark}

Another useful consequence of Gromov's Frobenius Theorem is the following proposition. We formulate it in the context of Cartan geometries; it was originally stated in \cite[§3]{gromov.rgs} for rigid geometric structures.

Let $(M,[g])$ be a manifold of dimension $\geq 3$ with a semi-Riemannian conformal structure. For $x \in M$, let $\Conf^{loc}_x$ denote the group of germs at $x$ of local conformal diffeomorphisms fixing $x$. Let $(M,\hm,\omega)$ be the normalized Cartan geometry modeled on $\mbox{\bf Ein}^{p,q}$ associated to $[g]$, and let $\hx \in \hm_x$. For all $f \in \Conf^{loc}_x$, there is a unique $p \in P$ such that $\hat{f}(\hx).p^{-1} = \hx$. The correspondence $\{f \in \Conf^{loc}_x \mapsto p \in P\}$ identifies $\Conf^{loc}_x$ with a subgroup $P^{\hx} \leq P$, called the \emph{isotropy of $f$ at $x$ with respect to $\hat{x}$}. 

\begin{proposition}
    \label{prop.isotropy-algebraic}
  For $(M,[g])$ real-analytic, the group $\Ad_{\lieg}(P^{\hx})$ is an algebraic subgroup of $\Ad_{\lieg}(P)$.
\end{proposition}

\begin{proof}
    Let $\Phi : \hm \to W$ be the $P$-equivariant map defined in \cite[Sec.4]{me.frobenius}. By Proposition 3.8 of the same article, $P^{\hx}$ coincides with the stabilizer in $P$ of $\Phi(\hx) \in W$. The result follows from the algebraicity of the action of $\Ad_{\lieg}(P)$ on $W$.
\end{proof}

\section{Linear mixed singularity implies conformal flatness}
\label{sec.mixed-flat}

Having gathered the necessary elements of the Cartan-geometric viewpoint on conformal structures, we can now pursue the program outlined in Section~\ref{ss.mixed-balanced}, with the aim of proving Theorem~\ref{thm.local.lichnerowicz}.  
The main result of this section is the following:

\begin{proposition}
\label{prop.mixed}
Let $(M,g)$ be a closed, real-analytic Lorentzian manifold. Let $Y$ be a conformal vector field on $M$. If $Y$ admits a singularity of mixed type, then $(M,g)$ is conformally flat. 
\end{proposition}

Recall the form of the linearization of a mixed-type singularity in (\ref{eqn.loxo.diffl}), with $|b|>|a|>0.$
Replacing $Y$ by $-Y$ if necessary, so that $a < 0$, the dynamics of $\{ \varphi^t_Y \}$ near the singularity $x_0$ exhibit a one-dimensional unstable manifold through $x_0$ and a codimension-one strongly stable manifold through $x_0$.  The singularity $x_0$ is locally isolated.

\subsection{A general result about zeros of conformal vector fields}
\label{ss.conjugated-zeros}
The idea behind the proof of Proposition~\ref{prop.mixed} is to show that the existence of a mixed-type singularity necessarily entails the existence of contracting singularities, which in turn implies conformal flatness.  Such a proof appears for $M$ simply connected in \cite[Sec 5.2]{mp.confdambra}.
The following proposition is a generalization of that result:

\begin{proposition}
\label{prop.conjugated}
Let $(M,g)$ be a smooth Lorentzian manifold. Let $\gamma: (-\epsilon,1+\epsilon) \to M$, with $\epsilon>0$, be a smooth immersion parametrizing a null geodesic segment. Let $Y$ be a conformal vector field vanishing at $\gamma(0)$ and $\gamma(1)$, such that $Y(\gamma(t))$ is collinear with $\gamma'(t)$ for all $t$. If $\gamma(0)$ is a singularity of mixed type for $Y$, then there exists $0 < t_0 \leq 1$ such that $\gamma(t_0)$ is a contracting singularity of $Y$ or of $-Y$.
\end{proposition}

\begin{proof}
The proof will use the notation and basic facts from Sections \ref{sec.cartan} and \ref{sec.algebraic-facts}. 

Let $(\widehat{M},\omega)$ be the normal Cartan geometry corresponding to $(M,[g])$. 
Because the singularity $\gamma(0)$ of $Y$ is isolated, there exists $0 < t_0 \leq 1$ such that $Y$ vanishes at $\gamma(t_0)$, but $Y(\gamma(t)) \neq 0$ for $t \in (0,t_0)$. 

Lift $Y$ to a vector field $\hat{Y}$ on $\widehat{M}$ satisfying $L_{\hat{Y}} \omega = 0$.  
There is $\hat{x}_0 \in \widehat{M}_{\gamma(0)}$ such that
\[
\omega(\hat{Y}(\hat{x}_0)) = \operatorname{diag}(a,b,R,-b,-a) =: D,
\]
with $b>|a|>0$ and $R \in \mathfrak{o}(n-2)$.
   Now lift $\gamma$ to an immersion $\hat{\gamma}: (-\epsilon,1+\epsilon) \to \widehat{M}$ such that $\hat{\gamma}(0) = \hat{x}_0.$ 
   
Denote $\xi(t) := \omega(\hat{Y}(\hat{\gamma}(t)))$.  
The hypothesis that $Y$ is collinear with $\gamma'$ means that
\[
\xi(t) = \lambda(t)\,\omega(\hat{\gamma}'(t)) + v(t),
\]
where $\lambda: (-\epsilon,1+\epsilon) \to \BR$ and $v(t) \in \mathfrak{p}$ for all $t \in (-\epsilon,1+\epsilon)$. Evaluating the Killing transport equation (\ref{eq.killing.transport}) and using that $\kappa$ is antisymmetric and semibasic 
gives:
\begin{equation}
 \label{eq.killing.transport.plat}
 \xi'(t)= [\xi(t),\omega(\hat{\gamma}'(t))].
\end{equation}
Let $\hat{\gamma}_*$ denote the development of $\hat{\gamma}$ in $\operatorname{PO}(2,n)$ (see Section \ref{sec.lightlike-geodesic}), satisfying
\[
\omega^G(\hat{\gamma}_*'(t)) = \omega(\hat{\gamma}'(t)) \quad \text{for all } t \in (-\epsilon,1+\epsilon),
\]
with $\hat{\gamma}_*(0)=1_G$. Since $\gamma$ is assumed to be an immersion, the vector $\hat{\gamma}'(t)$ is never vertical. The same is therefore true for $\hat{\gamma}_*'(t)$, and $\gamma_* := \pi_G \circ \hat{\gamma}_*$ is also an immersion in the flat model $(\PO(2,n),\omega^G)$. 
Now $t \mapsto \gamma_*(t)$ parametrizes a segment of a photon $\Delta \subset \Ein$ (see Section \ref{sec.lightlike-geodesic}), and $\Delta$ contains $\gamma_*(0)=[e_0]$.

Let ${\hat Y}_*$ be the right-invariant vector field on $\operatorname{PO}(2,n)$ with ${\hat Y}_*(1_G)=D$. Because $L_{{\hat Y}_*}\omega^G=0$ and the Cartan curvature $\kappa^G \equiv 0$, the function
\[
\xi_*(t) := \omega^G(Y_*(\hat{\gamma}_*(t)))
\]
satisfies 
\[
\xi_*'(t) = [\xi_*(t),\omega^G(\hat{\gamma}_*'(t))] = [\xi_*(t),\omega(\hat{\gamma}'(t))].
\]
Thus $\xi$ and $\xi_*$ satisfy the same first-order ODE and have the same value at $t=0$, so $\xi(t)=\xi_*(t)$ for all $t \in (-\epsilon,1+\epsilon)$. Let $Y_*$ be the projection of ${\hat Y}$ to $\Ein$. Now $Y_*(\gamma_*(t))$ is tangent to $\gamma_*'(t)$ for all $t \in (-\epsilon,1+\epsilon)$, and $Y_*(\gamma_*(t))=0$ if and only if $Y(\gamma(t))=0$. In particular $Y_*(\gamma_*(t))\neq 0$ for $t \in (0,t_0)$, and $Y_*(\gamma_*(t_0))=0$. The nature of the singularity $\gamma(t_0)$ corresponds to the $P$-conjugacy class of $\xi(t_0)$, which equals that of $\xi_*(t_0)$.

Observe that there are only two photons in $\Ein$ containing $[e_0]$ which are invariant under the matrix $D$, namely $\Delta_1=[\operatorname{span}(e_0,e_1)]$ and $\Delta_n=[\operatorname{span}(e_0,e_n)]$. Hence $\Delta$ equals either $\Delta_1$ or $\Delta_n$. Consider the case $\Delta=\Delta_1$. The vector field $Y_*$ has exactly two singularities on $\Delta_1$, namely $[e_0]$ and $[e_1]$, so $\gamma_*(t_0)$ equals $[e_0]$ or $[e_1]$. Since $Y_*(\gamma_*(t)) \neq 0$ for $t \in (0,t_0)$ and $\gamma_*$ is an immersion, $\gamma_*(t_0)=[e_1]$. 

Let
$$r_1 =
\left(
\begin{array}{ccccc}
0 & 1 &  &   &   \\
1 & 0 &  &  &  \\
  &  &  \mbox{Id}_{n-2} &  &  \\
  &  &   & 0 & 1 \\
  &  &   & 1 & 0 
  \end{array}
  \right)
\in \operatorname{PO}(2,n)$$ 

As $\pi_G(r_1)=[e_1]$, there exists $p_1 \in P$ such that $\gamma_*(t_0)\cdot p_1 = r_1$. As a consequence,
\[
\omega^G({\hat Y}_*(\hat{\gamma}_*(t_0))) = \left( \operatorname{Ad}(p_1)\circ \omega^G \right) ({\hat Y}_*(r_1)).
\]
By right-invariance of $\hat{Y}_*$ and equivariance of $\omega^G$, 
\[
\omega^G({\hat Y}_*(r_1)) = \omega^G(R_{r_1*}(\hat{Y}_*(1_G))) = \operatorname{Ad}(r_1^{-1}) (  D) = \operatorname{diag}(b,a,R,-a,-b).
\]
Hence
\[
\omega(\hat{Y}(\gamma(t_0))) = \operatorname{Ad}(p_1)\cdot\operatorname{diag}(b,a,R,-a,-b).
\]
Since $\operatorname{diag}(b,a,R,-a,-b) \in \mathfrak{g}_0$, the field $Y$ is linearizable at $\gamma(t_0)$. Because $b>|a|$, the singularity $\gamma(t_0)$ is contracting as claimed.

When $\Delta=\Delta_n$, we proceed along the same lines, replacing $r_1$ with $$r_n = 
\left(
\begin{array}{ccccc}
 &  &  & 1  & 0  \\
 &  &  & 0  & 1 \\
  &  &  \mbox{Id}_{n-2} &  &  \\
1  & 0 &   &  &  \\
0  & 1 &   &  &  
  \end{array}
  \right)
\in \operatorname{PO}(2,n)
$$
In this case one obtains
\[
\omega(Y(\gamma(t_0))) = \operatorname{Ad}(p_n)\cdot \operatorname{diag}(-b,-a,R,a,b),
\]
 for some $p_n \in P$, which proves that   the singularity $\gamma(t_0)$ is expanding, since $-b<0$ and $b>|a|$. It is thus contracting for $-Y$.
\end{proof}

\subsection{Proof of Proposition \ref{prop.mixed}}
Let $x_0$ be the mixed-type singularity of $Y$. Replacing $Y$ with $-Y$ if necessary, there exists a frame $(E_1, \ldots, E_n)$ in which the differential $D_{x_0} \varphi^t_Y$ takes the form given by (\ref{eqn.loxo.diffl}) with $b > a > 0$, and $\{R^t\} < \operatorname{SO}(n-2)$. Let $U$ be a neighborhood of $x_0$ and $g_0 \in [\left. g \right|_U]$ be as in Lemma~\ref{lem.linearizable-homothetic}, so that $\{\varphi^t_Y \}$ acts by homothetic transformations of $g_0$.
A local parametrization of the stable manifold through $x_0$ is given by $s \mapsto \exp_{{x}_0}(sE_n)$, where the exponential map is that of the metric $g_0$.  
The local stable manifold is a \emph{null geodesic segment} (see Subsection \ref{sec.lightlike-geodesic}).  

Let $s > 0$ be small enough that $x_{s} := \exp_{x_0}(s E_n)$ is defined, and introduce
$$
\alpha := \left\{ \varphi^t_Y \cdot x_{s} \;\middle|\; t \in \BR \right\} \subseteq \Delta
$$
Now consider the enhanced geometric structure $[g] \cup \{Y\}$, as in Section \ref{sec.stratification}.  Any conformal vector field defined locally around $x_0$ and commuting with $Y$ must vanish at $x_0$ and be tangent to the local stable manifold of $\{ \varphi^t_Y \}.$
Thus the null geodesic segment $\alpha$ coincides with the $\Kill^{loc}_Y$-orbit of $x_{s}$.  
Since this $\Kill^{loc}_Y$-orbit is one-dimensional, Theorem~\ref{thm.stratification} gives a clear picture of the closure $\overline{\alpha}$ of $\alpha$ in $M$. We refer the reader to the proof of Lemma~5.4 in \cite{mp.confdambra} to see how Theorem~\ref{thm.stratification} leads to:

\begin{lemma}
\label{lem.orbit-closure}
There exists a point $x_1 \in M$ such that:
$$
\overline{\alpha} = \{x_0\} \cup \alpha \cup \{x_1\}.
$$
and $\lim_{t \to -\infty} \varphi^t_Y \cdot x_{s} = x_1$.
\end{lemma}

Note that $x_1$ is a singularity of $Y$, and that \emph{a priori} it may coincide with $x_0$.  
The next goal is to find an immersion  
$\gamma: (-\epsilon, 1 + \epsilon) \to M$, for $\epsilon > 0$, parametrizing a null geodesic segment with 
$\gamma(0) = x_0$, $\gamma(1) = x_1$, and  
$\gamma([0,1]) = \{x_0\} \cup \alpha \cup \{x_1\}$.  
For $T \in \BR$, introduce the set:
$$
\alpha_{\geq T} := \left\{ \varphi^t_Y \cdot \exp_{x_0}(s_0 E_n) \;\middle|\; t \geq T \right\},
$$
and
similarly for 
$\alpha_{\leq T}.$
For large $T > 0$, the set  
$\{x_0\} \cup \alpha_{\geq T}$ is an interval of a null geodesic segment.
 
To prove the same for $\alpha_{\leq -T} \cup \{x_1\}$ 
fix a metric $g \in [g]$, as well as an auxiliary Riemannian metric $h$ on $M$, which we will use to define metric balls, either in tangent spaces $T_zM$ with respect to $h_z$, or in $M$ with respect to the Riemannian distance $d_h$.  
There exist constants $R > 0$ and $r > 0$, and a neighborhood $V$ of $x_1$, such that if $z \in V$, the exponential map $\exp_z$ with respect to $g$ is defined and injective on $B(0_z, R)$, 
and $\exp_z(B(0_z, R))$ contains the closed ball $\overline{B}(z, r)$.
Let $z_T := \varphi^{-T} (x_{s})$.  
For large $T > 0$, we have $z_T \in U$ and  
$\alpha_{\leq -T} \cup \{x_1\} \subset B(z_T, r)$.  
Let $u$ be a tangent vector to $\alpha$ at $z_T$, of unit $h$-norm.  
Then the map
\[
s \mapsto \exp_{z_T}(s u), \quad s \in (-R, R),
\]
where the exponential is again taken with respect to $g$, is an immersion that parametrizes a null geodesic segment containing  
$\alpha_{\leq -T} \cup \{x_1\}$ in its interior.

Proposition~\ref{prop.conjugated} now says that the field $Y$, or $-Y$, must admit another singularity of contracting type.  
The conformal flatness of $(M, [g])$ follows from Corollary~\ref{coro.contraction}.

\section{Linear balanced singularity implies local gravitational pp-wave metric in conformal class}
\label{sec.pp-wave}

This is the first of two sections treating balanced singularities of a conformal Lorentzian vector field (as defined in Subsection \ref{ss.taxonomy}).  The main result of this section, Theorem~\ref{thm.pp-wave}, gives local existence of very special metrics in the conformal class.  It is valid for non-compact real-analytic manifolds, and is of independent interest; thus \emph{$M$ is not assumed compact for the duration of this section}.  Such metrics have been obtained before in similar situations---see, for example, \cite{frances.degenerescence}, \cite[Sec 4]{mp.confdambra}, and \cite{mn.tractors}.  We provide a self-contained proof here of the precise and detailed statement that will be used to complete our proof of Theorem~\ref{thm.local.lichnerowicz}
(which will require compactness, as explained in Section \ref{sec.intro} above).

For $Y$ a conformal vector field vanishing at $x_0$, which is linearizable in a neighborhood $U$ of $x_0$ and of balanced type, the form of the linearization is given by (\ref{eqn.loxo.diffl}), with $|b|=|a| \neq 0$.
Conjugating in $\mbox{CO}(1,n-1)$ and multiplying $Y$ by a constant if necessary, we can assume $a = b = 1$.
Observe that although $Y$ may not be complete, $\varphi^t_Y$ is \emph{defined for all times $t \geq 0$} in $U$. The semi-group $\{ \varphi^t_Y \}_{t \geq 0}$ has a one-dimensional manifold of fixed points in $U$ and a transverse foliation by codimension-one strongly stable manifolds.  In particular, for every $x \in U$, $\lim_{t \to +\infty} \varphi^t_Y(x)$ exists and belongs to the fixed set in $U$. 

A \emph{Brinkmann metric} is a Lorentzian  metric  admitting a parallel lightlike vector field $X$. Consequently, the orthogonal distribution $X^{\perp}$ is invariant by the Levi-Civita connection of the Lorentz metric and therefore integrable. In the case when the foliation defined by $X^{\perp}$ has flat leaves with respect to the restriction of the Levi-Civita connection, the Brinkmann metric is a 
{\it pp-wave metric}. When a pp-wave metric is moreover Ricci-flat, it is called a {\it gravitational pp-wave}.
The notion of {\it polarization} for a Lorentzian manifold $(M,g)$ of dimension $\geq 4$ is defined with respect to the Weyl tensor $W$ on $M$. Let ${\mathcal D}$ be a lightlike one-dimensional smooth distribution on $M$. Then $M$ is \emph{polarized with respect to ${\mathcal D}$} if $W({\mathcal D}^{\perp},{\mathcal D}^{\perp},{\mathcal D}^{\perp})=0$ and $\operatorname{Im}W \subset {\mathcal D}^{\perp}$. 

Let $Q < P$ denote the stabilizer of the line $\BR E_1 \subset \lieg_{-1}$ under the adjoint action.  In the notation of (\ref{eqn:root_decomposition}) above,  the Lie algebra $\lieq$ of $Q$ is $\lieq=\liea \oplus \liem \oplus \lieg_{\beta} \oplus \lieg_{\alpha + \beta}$. The aim of this section is to prove the following result:

\begin{theorem}
\label{thm.pp-wave}
Let $(M,g)$ be a real-analytic Lorentzian manifold of dimension $\geq 3$. 
Assume that $Y$ is a conformal vector field,  admitting a linearizable singularity $x_0$ which is balanced. Then
\begin{enumerate}
    \item If the dimension of $M$ is $3$, then $(M,g)$ is conformally flat.
    \item In any dimension $\geq 4$, there exists a neighborhood $V$ of $x_0$ such that the conformal class $[\left. g \right|_V]$ contains a gravitational pp-wave metric $g_0$, for which $Y$ is a homothetic vector field.
    \item If moreover $(M,g)$ is not conformally flat (in particular, $M$ is of dimension $\geq 4$), there is a unique $\operatorname{Conf}^{loc}(M)$-invariant, analytic lightlike line field $\mathcal{D}$ on $M$, with the following properties:
    \begin{enumerate}
    \item The distribution $\mathcal{D}^{\perp}$ is integrable, and $Y$ is tangent to the leaves of $\mathcal{D}^{\perp}$. Moreover, at  $x_0$, the direction $\mathcal{D}$ is transverse to the local one-dimensional singular locus of $Y$.
        \item $(M,g)$ is polarized with respect to $\mathcal{D}$.
        \item In a neighborhood of each point of $M$, there is nonvanishing section $X$ of $\mathcal{D}$ and a local gravitational pp-wave metric for which $X$ is parallel and is uniquely determined by this property up to scalar multiplication. This class of local conformal vector fields determines an analytic, $\Conf^{loc}(M)$-invariant, global reduction of $\hm$ to $Q$.
        \item Up to multiplying $Y$ by a scalar, for every local conformal vector field $X$ as in (c), $[Y,X]=2X$.       
    \end{enumerate}
\end{enumerate}
\end{theorem}

\begin{remark}
    As the proof will show, point (2) does not require analyticity of $(M,g)$, hence is valid for Lorentzian metrics with $\mathcal{C}^{\infty}$ regularity.
\end{remark}

Point (1) of the proposition, dealing with the three-dimensional case, follows quickly from previous results.  Consider the linearizing coordinate system $(x_1,x_2,x_3)$ for $U$. For every $x \in U$, the matrix of $D_x \varphi^t_Y$, $t \geq 0$, relatively to the frame field $(\frac{\partial}{\partial x_1}, \frac{\partial}{\partial x_2}, \frac{\partial}{\partial x_3})$ is $\operatorname{diag}(e^{-2t},e^{-t},1)$.
It follows that, given $t_k \to + \infty$, the sequence $\{ \varphi^{t_k}_Y \}$ is stable at $x$, in the sense of \cite[Def. 3]{frances.ccvf}. Now \cite[Prop. 5]{frances.ccvf} yields that the Cotton tensor vanishes on $U$. By analyticity, $(M,g)$ is conformally flat. Point (1) is proved.

\subsection{Parallel submanifolds in the Cartan bundle and special metrics in the conformal class}
\label{sec.parallel}
As above, let $(\hm,\omega)$ be the normal Cartan geometry modeled on $\Ein$ associated to the conformal Lorentzian structure $(M,[g])$ as given by Theorem \ref{thm.cartan}.  The notation of Section \ref{sec.cartan} and Section \ref{sec.algebraic-facts} will again be used below.

Given a vector subspace $\mathfrak{h} \subset \mathfrak{g}$, a submanifold $\hat{\Sigma} \subset \hm$ is \emph{parallel with respect to $\mathfrak{h}$} when $\hat{\Sigma}$ is an integral leaf of the distribution $\omega^{-1}(\mathfrak{h})$.
Equivalently, for every $\hat{x} \in \hat{\Sigma}$, one has $\omega_{\hat{x}}(T_{\hat{x}} \hat{\Sigma}) = \mathfrak{h}$.
The most interesting situation occurs when $\mathfrak{h}$ is a Lie subalgebra of $\mathfrak{g}$, for in that case  $\omega$ induces a Cartan subgeometry on $\hat{\Sigma}$.
The presence of parallel submanifolds of dimension $>1$ often reflects special geometric properties of the conformal structure $(M,[g])$. For instance, if, for a conformal structure, the distribution $\omega^{-1}(\mathfrak{\lieg}_{-1})$ admits an integral leaf through $\hat{x} \in \hm$, then projecting this leaf yields a conformally flat neighborhood of $x=\pi(\hx)$.  Similarly, an integral leaf of $\omega^{-1}(\lieg_{-1} \oplus \lieg_0)$ yields a local Ricci-flat metric in the conformal class around $x$.
This fact and further refinements of it are recorded in the following:

\begin{proposition}
\label{prop.brinkmann}
Let $(M,[g])$ be a smooth conformal Lorentzian manifold. 
\begin{enumerate}
    \item Assume that there exists an integral leaf $\hat{\Sigma} \subset \hm$ of the distribution $\omega^{-1}(\lieg_{-1} \oplus \lieu^+)$, and let  $\hat{x} \in \hat{\Sigma}$. Then there exists, in a neighborhood $V$ containing $x=\pi(\hat{x})$,  a Ricci-flat Brinkmann metric in the conformal class $[\left. g \right|_V]$ with parallel lightlike vector field equal to the projection of $\left. \omega \right|_{\hat{\Sigma}}^{-1}(E_1)$.
    \item If moreover, the Cartan curvature satisfies $\kappa_{\hx}(\zeta,\eta)=0$ for every $\hx \in \hat{\Sigma}$ and  $\zeta,\eta$ in $E_1^{\perp}$, the Ricci-flat Brinkmann metric of $(1)$ can be assumed to be a gravitational $pp$-wave metric.
\end{enumerate}
\end{proposition}

\begin{proof}
Recall the conformal exponential map from Section \ref{sec.exponential}. 

Because $\hat{\Sigma}$ is parallel with respect to $\lieg_{-1} \oplus \lieu^+$,  there exists a neighborhood $\mathcal{V}$ of $0$ in $\lieg_{-1} \oplus \lieu^+$,  for which $\exp(\hat{x}, \mathcal{V})$ is defined and such that $\exp(\hat{x}, \mathcal{V}) \subset \hat{\Sigma}$. We may choose $\mathcal{V}$ small enough that  $\hat{V}_{-1}= \exp(\hat{x}, \mathcal{V} \cap \lieg_{-1})$ is transverse to the fibers of $\pi$   and projects diffeomorphically   onto its image $V$, an open neighborhood of $x$.  Then consider  $\hat{V}=\hat{V}_{-1} \cdot U^+$, the saturation of $\hat{V}_{-1}$  by the $U^+$-action. It is a $U^{+}$-principal bundle over  $V$.  Observe that $\hat{V}$ is still a parallel submanifold with respect to $\lieg_{-1} \oplus \lieu^+$. Indeed, $\omega(T\hat{V}_{-1}) \subset \omega(T\hat{\Sigma}) \subset \lieg_{-1} \oplus \lieu^+$; in fact, by $U^+$-equivariance of $\omega$ and $\Ad(U^+)$-invariance of $\lieg_{-1} \oplus \lieu^+$, the subspaces $\omega(T\hat{V}) \subset \lieg_{-1} \oplus \lieu^+$. 

The bundle $\hat{V}$ defines a metric $g_0$ in the conformal class $[\left. g \right|_V]$ as in (\ref{eq.conformal-class}): namely, for $y \in V$, define:
\begin{equation}
\label{eq.formula}
g_0(u,v):= \mathbb{I}( \iota_{\hy}(u), \iota_{\hy}(v)) \qquad \forall \ u,v \in T_yV 
\end{equation}
where $\hy$ is any point of $\hat{V}_y$, and $\mathbb{I}$ is the Lorentzian scalar product on $\lieg/\liep$ introduced in Sections \ref{sec.equivalence} and \ref{sec.algebraic-facts}.
The definition of $g_0$ is consistent: writing $\mathbb{I}(a,b) = \langle a, b \rangle$, equation (\ref{eq.equivariant-iota}) gives
$$ \langle \iota_{\hy.p}(u), \iota_{\hy.p}(v)\rangle= \langle \Ad(p^{-1})\iota_{\hy}(u), \Ad(p^{-1})\iota_{\hy}(v) \rangle= \langle \iota_{\hy}(u), \iota_{\hy}(v) \rangle \qquad \forall \ p \in U^+.$$

By restriction, the Cartan connection $\omega$ induces a Cartan connection on $\hat{V}$ with values in the Lie algebra $\lieg_{-1} \oplus \lieu^+$, which we will denote $\overline{\omega}$. The curvature $\overline{K}$ of $\overline{\omega}$ is  the restriction of $K$ to $T\hat{V}$.

A vector field on $\hat{V}$ is called \emph{horizontal} if it is tangent to $\overline{\omega}^{-1}(\mathfrak{g}_{-1})$. 
For any vector field $X$ on $V$, there exists a unique lift $\hat{X}$ to $\hat{V}$ which is horizontal and right-$U^+$-invariant. 
Conversely, any horizontal $U^+$-invariant vector field $\hat{X}$ on $\hat{V}$ projects to a well-defined vector field $X$ on $V$.

The formula
\begin{equation}
\label{eq.connection}
\widehat{\nabla_X Y} = \overline{\omega}^{-1} \big( \hat{X} . \overline{\omega}(\hat{Y}) \big)
\end{equation}
defines a connection $\nabla$ on $V$. 
This connection is torsion-free.  Indeed, for any $X, Y \in \mathcal{X}(V)$, 
\begin{equation}
\label{eqn.torsion}
\nabla_X Y - \nabla_Y X - [X,Y]  =  \pi_*(\widehat{\nabla_X Y} - \widehat{\nabla_Y X} - [\hat{X}, \hat{Y}] )
\end{equation}

Then
\begin{equation*}
\omega(\widehat{\nabla_X Y} - \widehat{\nabla_Y X} - [\hat{X}, \hat{Y}]) = \hat{X}.\omega(\hat{Y}) - \hat{Y}.\omega(\hat{X}) - \omega([\hat{X},\hat{Y}])
\end{equation*}
The right-hand side equals $K(\hat{X},\hat{Y})$ because $[\omega(\hat{X}), \omega(\hat{Y})] = 0$ by commutativity of $\lieg_{-1}$.
Now $K(\hat{X},\hat{Y}) = \overline{K}(\hat{X},\hat{Y})$ takes values in 
$(\mathfrak{g}_{-1} \oplus \mathfrak{u}^+) \cap \mathfrak{p} = \mathfrak{u}^+$, so (\ref{eqn.torsion}) equals $0$. 
 Writting $\langle \cdot, \cdot \rangle$ for the pullback of $\mathbb{I}$ from $\lieg/\liep$ to $\lieg$, we observe that for any three vector fields $X, Y, Z$ on $V$
\[
\hat{Z} . \langle \omega(\hat{X}), \omega(\hat{Y}) \rangle
= \langle \hat{Z} . \omega(\hat{X}), \omega(\hat{Y}) \rangle
+ \langle \omega(\hat{X}), \hat{Z} . \omega(\hat{Y}) \rangle,
\]
thus $Z . g_0(X,Y) = g_0(\nabla_Z X, Y) + g_0(X, \nabla_Z Y)$. 
Therefore, $\nabla$ is the Levi-Civita connection of $g_0$.

Let $\hat{E}_1$ be the vector field on $\hat{V}$ defined by $\overline{\omega}(\hat{E}_1) \equiv E_1$. 
Because $\Ad(U^+)$ fixes $E_1$, the field $\hat{E}_1$ is right-$U^+$-invariant and hence defines a vector field $X$ on $V$. 
Since $\langle E_1, E_1 \rangle = 0$, the field $X$ is lightlike. 
Moreover, for any $\hat{Y}$ horizontal and $U^+$-invariant, we have $\hat{Y}. \omega(\hat{E}_1) = 0$; 
therefore $\nabla X = 0$. Thus $X$ is parallel with respect to $g_0$, and $g_0$ is a Brinkmann metric.

Finally, for $R$ the curvature tensor of the metric $g_0$, for $y \in V,$ and $u, v, w \in T_yV$,
\[
R_y(u,v,w)
= \overline{\kappa}(\iota_{\hat{y}}(u), \iota_{\hat{y}}(v)). \iota_{\hat{y}}(w)
= W_y(u,v,w) 
\]
(See Section~\ref{sec.curvature} to recall the link between the Cartan curvature and the Weyl tensor.)
The Riemann tensor of $g_0$ is thus equal to its Weyl tensor, proving that $g_0$ is Ricci-flat.

 Now suppose the curvature hypothesis of point $(2)$ is satisfied and let $\hat{V}$ be as above. 
 Consider $Z_1$ and $Z_2$ two $\omega$-constant vector fields on $\hat{V}$. 
 By equation (\ref{eq.curvature}) defining the curvature, $Z_1$ and $Z_2$ satisfy the identity:
\begin{equation}
\label{eq.identity-curvature}
\omega([Z_1,Z_2])=[\omega(Z_1),\omega(Z_2)]-K(Z_1,Z_2).
\end{equation}
 If moreover $Z_1$ and $Z_2$ are tangent to $\left. \omega \right|_{\hat{V}}^{-1}(E_1^{\perp})$, then 
$$\omega([Z_1,Z_2])=[\omega(Z_1),\omega(Z_2)]=0.$$
The distribution  $\left. \omega \right|_{\hat{V}}^{-1}(E_1^{\perp})$ is thus integrable on $\hat{V}$. 
 With $g_0$ given by (\ref{eq.formula}),
let $X$ as above be the parallel lightlike vector field corresponding to $E_1$ in the frames of $\hat{V}$.
Let ${\mathcal D}$ be the line field generated by $X$.  Each integral leaf of $\left. \omega \right|_{\hat{V}}^{-1}(E_1^{\perp})$ projects to an integral leaf of $\mathcal{D}^{\perp}$ on $V$, so ${\mathcal D}^{\perp}$ is integrable. 
Because the leaves of $\left. \omega \right|_{\hat{V}}^{-1}(E_1^{\perp})$ are $\omega$-parallel by definition, (\ref{eq.connection}) shows that their projections to $V$ are $\nabla$-parallel, hence totally geodesic. The curvature of the induced connection vanishes because by the hypothesis of $(2)$, 
$$\overline{\kappa}(\iota_{\hat{y}}(u), \iota_{\hat{y}}(v)) = 0 \qquad \forall \ u,v \in \mathcal{D}^{\perp}.$$
Therefore the integral leaves of $\mathcal{D}^{\perp}$ are also flat. The metric $g_0$ is thus a Ricci-flat pp-wave.
\end{proof}

\subsection{Existence of a gravitational pp-wave metric in the neighborhood of $x_0$}
\label{ss.pp-wave}
This section provides the proof of Theorem \ref{thm.pp-wave} (2). 
Linearizability of the conformal flow $\{ \varphi^t_Y \}$ at $x_0$ means there is $\hx_0 \in \hm_{x_0}$ and $\{h^t\} < G_0$, such that $\varphi^t_Y(\hx_0).h^{-t}=\hx_0$ for all $t$ (see for instance \cite[Prop. 4.2]{frances.locdyn}).  The linear isomorphism $\iota_{\hx_0}$ conjugates the action of $D_{x_0}\varphi^t_Y$ on $T_{x_0}M$ to the action of $\Ad(h^t)$ on $\lieg/\liep$. Up to conjugating $\{ h^t \}$ in $G_0$, which corresponds to right-translating $\hx_0$ in the fiber, the action of $\Ad(h^t)$ on $\lieg/\liep \cong \lieg_{-1}$ is, as above, by the matrix $\bar{h}^t$ of the form (\ref{eqn.loxo.diffl}) with $a=b=1$, with respect to the basis $(E_1,\ldots,E_n)$ of Section \ref{sec.algebraic-facts}. In $\PO(2,n)$, the block-diagonal form of $h^t$ is :
\begin{equation}
    \label{eqn.ht}
h^t= \diag(e^t,e^{-t},R^t, e^t, e^{-t}),
\end{equation}
where $\{R^t\} < \OO(1,n-1)$ is the same bounded one-parameter subgroup appearing in $\bar{h}^t$.

\subsubsection{Parallel submanifolds determined by the dynamics}
\label{sec.parallel-dynamics}
 The adjoint action of $h^t$ on $\lieg=\operatorname{\mathfrak{so}}(2,n)$ preserves each root space. The $+1$-eigenspace of $\Ad h^t$ for all $t$ is the direct sum  
\begin{equation}
\label{eqn.centralizer.ht}
 \lies_0=\lieg_{-\alpha -\beta} +\liea + \mathfrak{m}+\lieg_{\alpha+\beta} = \{ \xi \in \lieg \ : \ \Ad(h^t)\xi = \xi \ \forall t \}
 \end{equation}
On $\lies_{\pm 1}=\lieg_{\pm \alpha}+\lieg_{\mp \beta}$, the action is by Euclidean similarities 
$$\xi \mapsto e^{\pm t} \Ad(R^t)\xi$$ 
Denote $\lies_{\pm 2}=\lieg_{\pm(\alpha - \beta)}$, the eigenspaces associated to the eigenvalues $e^{\pm 2t}$. The \emph{stable space} for $\Ad(h^t)$, $t \geq 0$, is
$$\lies^<=\lies_0+\lies_{-1}+\lies_{-2}=\lieg_{-1}+\liea+\liem+\lieu^{+}+\lieg_{\alpha+\beta},$$
comprising all vectors $\xi \in \lieg$
 such that $\Ad(h^t)\xi$ remains bounded for $t \geq 0$. The \emph{strongly stable space} is 
 $$\lies^{<<}=\lies_{-1}+\lies_{-2}=E_1^{\perp}+\lieu^+ = \{ \xi \in \lieg \ : \ \lim_{t \to + \infty } \Ad(h^t)\xi=0 \}$$
The distributions $\omega^{-1}(\lies^<)$ and  $\omega^{-1}(\lies^{<<})$ can be integrated to parallel submanifolds in $\hm$, by the following proposition (see also \cite[Sec 4]{mp.confdambra}). 

\begin{proposition}[\cite{frances.degenerescence}, Prop. 4.8]
\label{prop.variete-stable}
There exists a neighborhood $\mathcal{V}$ of $0$ in $\lieg$ such that $\hat{\Sigma}^{<}=\exp(\hx_0, \mathcal{V} \cap \lies^{<})$ and $\hat{\Sigma}^{<<}=\exp(\hx_0, \mathcal{V} \cap \lies^{<<})$ are integral leaves of the distributions $\omega^{-1}( \lies^{<})$ and $\omega^{-1}( \lies^{<<})$.
\end{proposition}

In the sequel, it will be useful to choose $\mathcal{V}$ as a product of convex neighborhoods of the origin in each rootspace. 

\subsubsection{Special values of the curvature function}
\label{sec.special-curvature}
 Recall from Proposition \ref{prop.brinkmann} that a local integral leaf of the distribution $\omega^{-1}(\lieg_{-1} + \lieu^+)$ corresponds to a local Brinkmann metric in the conformal class.
We will show that this distribution is integrable on the manifold $\Sigma^<$. Involutivity will follow from the special form of the Cartan curvature: 

\begin{lemma}
\label{lem.curvature-values}
For every $\hx \in \hat{\Sigma}^{<}$, the curvature $\kappa_{\hx}$ has the following properties:
\begin{enumerate}
\item $\operatorname{Im} \kappa_{\hx} \subset \lieu^+$.
\item For every $\xi,\eta$ in $E_1^{\perp}$, $\kappa_{\hx}(\xi,\eta)=0$.
\end{enumerate}
\end{lemma}

\begin{proof}
Let $\hx \in \hat{\Sigma}^{<}$. It is of the form $\hx=\exp(\hx_0,\xi)$, for $\xi \in \mathcal{V} \cap \lies^{<}$.  Writing $\xi = \xi^0 + \xi^{<<}$ with $\xi^0 \in \lies^0$ and $\xi^{<<} \in \lies^{<<}$,
the limit $\Ad(h^t)\xi \rightarrow \xi^0$ as $t \to + \infty$.
Equation (\ref{eq.equivariance-exponential}) yields:
$$ \hx_t:=\varphi^t_Y(\hx). h^{-t}=\exp(\hx_0,\Ad(h^t)\xi) \rightarrow \hx_* = \exp(\hx_0,\xi^0)$$
In particular, $\lim_{t \to + \infty}\kappa_{\hx_t}=\kappa_{\hx_*}$. The equivariance of $\kappa$, equation (\ref{eq.curvature}), gives $\kappa_{\hx_t}=h^t.\kappa_{\hx}$, from which we deduce that $h^t.\kappa_{\hx}$ remains bounded for $t \geq 0$. 
Referring to Section \ref{sec.parallel-dynamics} and (\ref{eqn.ht}) above, 
let $\xi,\eta$ each belong to one of the root spaces 
 $\lieg_{-\alpha+ i \beta}$, $i=-1,0,1$, and compute
 $$\kappa_{\hx}(\Ad(h^{-t}) \xi,\Ad(h^{-t}) \eta)=e^{\lambda t}\kappa_{\hx}(\Ad(R^{-t}) \xi, \Ad(R^{-t}) \eta),$$
 where $\lambda=1,2$ or $3$. Since $h^t.\kappa_{\hx}$ is bounded,  there exists $C \geq 0$ such that 
 \begin{equation}
 \label{eq.contraction}
 |\Ad(h^t). \kappa_{\hx}(\Ad(R^{-t}) \xi, \Ad(R^{-t}) \eta)| \leq Ce^{-\lambda t} \qquad \forall \ t\geq0
 \end{equation}

Because $\{ R^t\}$ is bounded, there is a subsequence $t_k \to \infty$ such that $R^{-t_k}\to R_*$. By Equation (\ref{eq.contraction}), the components of $\kappa_{\hx}(\Ad(R_*) \xi, \Ad(R_*)  \eta)$ on rootspaces on which $\Ad(h^t)$ does not act as a contraction must vanish.  The only rootspace in $\liep$ which is contracted by $\Ad(h^t)$ is $\lieu^{+}$, 
thus $\kappa_{\hx}(\Ad(R_*) \xi, \Ad(R_*) \eta)\in \lieu^{+}$.  The first point of Lemma \ref{lem.curvature-values} follows, because $\Ad(R_*)$ is an isomorphism on each rootspace.

To prove the second point of the lemma, observe that $E_1^{\perp}= \lieg_{-\alpha+\beta} \oplus \lieg_{-\alpha}$. 
Two vectors $\xi$ and $\eta$, each in $ \lieg_{-\alpha + i \beta}$ for $i=0,1$,
satisfy equation (\ref{eq.contraction}), but the possible values of $\lambda$ are $2$ or $3$. On the other hand, the action of $\Ad(h^t)$ on each rootspace of $\liep$ is by a transformation of the form $\xi \mapsto e^{\mu t} \Ad(R^t) \xi$, where $\mu \geq -1$. For the same sequence $\{ t_k \}$ as above, necessarily every rootspace component of $\kappa_{\hx}(\Ad(R_*) \xi, \Ad(R_*) \eta) = 0.$
This finishes the proof of Lemma \ref{lem.curvature-values}.
\end{proof}

\subsubsection{Construction of the gravitational pp-wave at $x_0$.}
\label{subsec.tot.geod.leaves}

Let $Z_1$ and $Z_2$ be $\omega$-constant vector fields on $\hat{\Sigma}^<$ tangent to $\omega^{-1}(\lieg_{-1}+\lieu^+)$. If $\omega(Z_1)$ or $\omega(Z_2)$ belongs to $\lieu^+$, then $K(Z_1,Z_2)=0$ and $\omega([Z_1,Z_2]) \in \lieg^{-1}+\lieu^+$ 
by (\ref{eq.identity-curvature}). If both $\omega(Z_1)$ and $\omega(Z_2)$ belong to $\lieg_{-1}$, then Lemma \ref{lem.curvature-values}(1) ensures that $\omega([Z_1,Z_2]) \in \lieu^+$. Therefore $\omega^{-1}(\lieg_{-1}+\lieu^+)$ is involutive on $\hat{\Sigma}^{<}$.  Moreover, Lemma \ref{lem.curvature-values} (2) gives that $\kappa_{\hx}(\xi,\eta)=0$ for every $\xi,\eta$ in $E_1^{\perp}$ and $\hx \in \hat{\Sigma}^{<}$. Let $\hat{V}$ be an integral leaf through $\hat{x}_0$.  It follows from Proposition \ref{prop.brinkmann} that, after shrinking $\hat{V}$ if necessary, there exists a neighborhood $V = \pi(\hat{V})$ of $x_0$, as well as a gravitational pp-wave metric in $[\left. g \right|_V]$,
and that the projection to $V$ of the restriction to $\hat{V}$ of $\hat{E}_1$, the $\omega$-constant vector field associated to $E_1$, is lightlike and parallel with respect to $g_0$.

\subsubsection{The field $Y$ is homothetic for $g_0$.}
\label{subsec.homothetic}
Now we can express $\hat{V}$, the integral leaf through $\hat{x}_0$, as 
$\exp(\hx_0, {\mathcal V}_1) \cdot U^+$, where ${\mathcal V}_1$ is a small neighborhood of $0$ in $\lieg_{-1}$. 
Taking ${\mathcal V}_1$ as a product of balls in $\lieg_{-\alpha+\beta}$, $\lieg_{-\alpha}$, and $\lieg_{-\alpha-\beta}$, we can ensure that $\Ad(h^t)({\mathcal V}_1) \subset {\mathcal V}_1$ for all $t \geq 0$. 

Let $\hy \in \hat{V}$, 
and write $\hy=\exp(\hx_0,\xi) \cdot u$ for some $\xi \in {\mathcal V}_{-1}$ and $u \in U^+$. 
Then 
\[
\varphi^t_Y (\hy) \cdot h^{-t} = 
\varphi^t_Y(\exp(\hx_0,\xi) \cdot u) \cdot h^{-t}
  = \exp(\hx_0, \Ad(h^t)\xi) \cdot h^t u h^{-t},
\]
which is an element of $\hat{V}$ because $\Ad(h^t)\xi \in {\mathcal V}_{-1}$ and $h^t u h^{-t} \in U^+$.
Let $y \in V$ and $\hy \in \hat{V}_y$. Given $u,v \in T_yM$, set $g_0(u,v)=\langle \iota_{\hy}(u), \iota_{\hy}(v) \rangle$ as above.  Because $\varphi^t_Y(\hy) \cdot h^{-t}$ is also in $ \hat{V}$, (\ref{eq.formula}) gives
\[
g_0(D_y\varphi^t_Y(u),D_y\varphi^t_Y(v))
 = \langle \iota_{\varphi^t_Y(\hy) \cdot h^{-t}}(D_y\varphi^t_Y(u)), 
    \iota_{\varphi^t_Y(\hy) \cdot h^{-t}}(D_y\varphi^t_Y(v)) \rangle.
\]
From (\ref{eq.equivariant-iota}) and $\varphi^t_Y$-invariance of $\omega$,
$$\iota_{\varphi^t_Y(\hy) \cdot h^{-t}}(D_y\varphi^t_Y(u))
  = \Ad(h^t)\iota_{\varphi^t_Y(\hy)}(D_y\varphi^t_Y(u)) = \Ad(h^t) \iota_{\hy}(u)$$
Thus
\[
g_0(D_y\varphi^t_Y(u), D_y\varphi^t_Y(v))
  = \langle \Ad(h^t)\iota_{\hy}(u), \Ad(h^t)\iota_{\hy}(v)\rangle
  = e^{2t}\langle \iota_{\hy}(u), \iota_{\hy}(v) \rangle
  = e^{2t}g_0(u,v).
\]
The second point of Theorem \ref{thm.pp-wave} is now fully proved.

\subsection{Third point of Theorem \ref{thm.pp-wave}: Existence of a global polarization}
\label{sec.polarization}

We assume now that $(M,g)$ {\it is not conformally flat}, which, by the first point of Theorem \ref{thm.pp-wave}, implies that $M$ has dimension $\geq 4$. 

According to the following lemma, $(M,g)$ can be polarized with respect to at most one lightlike line field.

\begin{lemma}[\cite{pecastaing.can.sl2r}, Lemma 9]
\label{lem.polarization}
Let $(M,g)$ be a real-analytic Lorentzian manifold of dimension $\geq 4$. Assume that $M$ is polarized with respect to distinct lightlike line fields ${\mathcal D}$ and ${\mathcal D}'$. Then $(M,g)$ is conformally flat.
\end{lemma}

(In fact, Lemma~9 of \cite{pecastaing.can.sl2r} establishes a pointwise result: if a Lorentzian manifold is polarized with respect to two distinct lightlike lines ${\mathcal D}_x$ and ${\mathcal D}'_x$ at a point $x$, then the Weyl tensor vanishes at $x$.) 

\subsubsection{Local conformal vector field and polarization.}
\label{subsec.polarization}

Let $\hat{V} \rightarrow V$ be as above, with $X$ the lightlike vector field on $V$ equal the projection of $\left. \omega \right|_{\hat{V}}^{-1}(E_1)$ and ${\mathcal D}$ the line field generated by $X$.

For $x \in V$ and $\hx \in \hat{V}_x$, Lemma \ref{lem.curvature-values} (1) ensures that $\kappa_{\hx}(\iota_{\hx}(u), \iota_{\hx}(v)) \in \lieu^{+}$ for all $u,v$ are in $T_xM$. Thus $\kappa_{\hx}(\iota_{\hx}(u), \iota_{\hx}(v)). \iota_{\hx}(w) \in E_1^{\perp}$ for all $w \in T_xM$. In other words, according to the interpretation of the Weyl curvature in terms of $\kappa$ from Section \ref{sec.geometric}, the image of $W_x$ is contained in $\mathcal{D}^{\perp}$ for all $x \in V$, and the conformal structure is polarized with respect to $\mathcal{D}$ on $V$.

Because $X$ is parallel for $g_0$, it is a Killing vector field of $g_0$, and in particular a conformal vector field on $V$. 
The theorem of Amores \cite{amores.killing} says that, after lifting $X$ to a diffeomorphic neighborhood $\tilde{V} \cong V$ in the universal cover $\tilde{M}$ of $M$, it extends to a global conformal vector field on $\tilde{M}$, which we denote $\tilde{X}$. As $\tilde{X}$ is lightlike on $\tilde{V}$, it is lightlike on all of $M$ by analyticity. 
Since $M$ is not conformally flat, Theorem 1 of \cite{frances.ccvf} implies that $\tilde{X}_p \neq 0$ for all $p \in \tilde{M}$. 

Let $\tilde{\mathcal{D}}$ be the analytic lightlike line distribution generated by $\tilde{X}$. The conformal structure on $\tilde{V}$ is polarized with respect to $\tilde{\mathcal{D}}.$ Now, $\tilde{\mathcal{D}}^\perp$ is an analytic distribution on $\tilde{M}$, and the conditions
$$ W(\tilde{\mathcal{D}}^\perp, \tilde{\mathcal{D}}^\perp, \tilde{\mathcal{D}}^\perp) = 0 \qquad \mbox{Im } W \subseteq \tilde{\mathcal{D}}^\perp$$
correspond locally to analytic equations on $\tilde{M}$. Therefore, the conformal structure is polarized with respect to $\tilde{\mathcal{D}}$ on all of $\tilde{M}$.  

Let $\pi_1(M) \cong \Gamma < \Conf(\tilde{M})$ be the deck transformations of $\tilde{M} \rightarrow M$.  For each $\gamma \in \Gamma$, conformal invariance of $W$ implies that $\tilde{M}$ is polarized with respect to $\gamma_* \tilde{\mathcal{D}}.$  By Lemma \ref{lem.polarization}, $\gamma_* \tilde{\mathcal{D}} = \tilde{\mathcal{D}}$ for all $\gamma \in \Gamma.$  Thus $\tilde{\mathcal{D}}$ descends to a line field on $M$ extending $\mathcal{D}$, with respect to which $M$ is polarized.  The extended line field will still be denoted $\mathcal{D}$.
By the same reasoning with Lemma \ref{lem.polarization}, the line field $\mathcal{D}$ is $\Conf^{loc}({M})$-invariant. 

Note, from the explicit linearization $\bar{h}^t$ of $\varphi^t_Y$ near $x_0$, that $\mathcal{D}$ is transverse to the singular locus of $Y$ at $x_0$. Indeed, $\mathcal{D}_{x_0}$ is spanned by the projection $\pi_*(\left. \omega \right|_{\hat{V}}^{-1}(E_1))$ while the tangent to the singular locus is spanned by $\pi_*(\left. \omega \right|_{\hat{V}}^{-1}(E_n))$.

\subsubsection{The field $Y$ is contained in $\mathcal{D}^{\perp}$}
\label{subsec.tangent}

Let $\hat{x}_0 \in \hat{V}$, and let $\hx_s:=\exp(\hx_0,sE_n)$ for $|s| < \epsilon$.  For $\epsilon$ sufficiently small, $\hx_s \in \hat{V}$ for all $s \in (-\epsilon,\epsilon)$, because $\hat{V}$ is parallel with respect to $\lieg_{-1}+ \lieu^+$. Let $\mathcal{W}$ be a neighborhood of $0$ in $E_1^{\perp}$, such that $\exp(\hx_s,\xi)$ is defined for all $s \in (-\epsilon,\epsilon)$ and $\xi \in \mathcal{W}$. Since $\left. \omega \right|_{\hat{V}}^{-1}(E_1^{\perp})$ is integrable on $\hat{V}$, the sets $\hat{F}_s:=\exp(\hx_s,\mathcal{W})$ are integral leaves of $\left. \omega \right|_{\hat{V}}^{-1}(E_1^{\perp})$ contained in $\hat{V}$. The Inverse Function Theorem implies that, shrinking $\epsilon$ and $\mathcal{W}$ if necessary, the map $\Phi: (-\epsilon,\epsilon) \times \mathcal{W} \to M$ defined by $\Phi(s,\xi):=\pi(\exp(\hx_s,\xi))$ is a diffeomorphism onto an open subset of $V$, which we assume to be $V$ to simplify notation. The submanifolds $F_s:=\pi(\hat{F}_s)$ foliate $V$, and are integral leaves of ${\mathcal D}^{\perp}$. For $y\in F_s$, lift $y$ to $\hy=\exp(\hx_s,\xi)$, with $\xi \in \mathcal{W}$. The equivariance of the exponential map (\ref{eq.equivariance-exponential}), and the property $\Ad(h^t)E_n=E_n$  yield
\[ \varphi^t_Y(\hx_s).h^{-t}=\exp(\hx_0, s\Ad(h^t)E_n)=\hx_s .\]
Now $\varphi^t_Y(\hy).h^{-t}=\exp(\hx_s,\Ad(h^t)\xi)$, and $\Ad(h^t)\xi \in \mathcal{W}$ for small $t$, since $E_1^{\perp}$ is preserved by $\Ad(h^t)$.
We conclude that $\varphi^t_Y(\hy).h^{-t} \in \hat{F}_s$, which implies $\varphi^t_Y(y) \in F_s$.
Thus $Y$ is tangent to $F_s$, which means it is contained in $\mathcal{D}^\perp$ inside $V$. By analyticity, $Y$ is contained in $\mathcal{D}^\perp$ at every point of $M$.
Points 3(a) and 3(b) of Theorem \ref{thm.pp-wave} are proved.

\subsection{Local gravitational pp-waves and Cartan bundle reduction}
\label{subsec.gravitatonal-everywhere}

This section will complete the proofs of points $(c)$ and $(d)$ of Theorem \ref{thm.pp-wave} (3). 
Denote ${\bf P}(\lieg)$ the projectivization of $\lieg$, and for $0 \neq u \in \lieg$, denote $[u] \in {\bf P}(\lieg)$ the projection of $u$. Let $X$ be the local conformal vector field on $V$ as constructed in Section \ref{subsec.polarization}, and $\hat{X}$ its lift to the subbundle $\hat{V}$.  Recall that $X$ was defined as $( \left. \pi \right|_{\hat{V}})_* (\hat{E_1})$, where $\omega(\hat{E_1}) \equiv E_1.$

\begin{lemma}
    \label{lem.valeurs-X}
    For every $\hx \in \hat{V}$, we have $\omega(\hat{X}_{\hx})=E_1$.
\end{lemma}
\begin{proof}
   For $\hx \in \hat{V}$, let $\omega(\hat{X}_{\hx})=E_1+\xi_0(\hx)+ \xi_{+1}(\hat{x})$, where $\xi_0(\hx)\in \lieg_0$ and $\xi_{+1}(\hx) \in \lieg_{+1}$.  
   The following two elementary algebraic properties of the Lie algebra $\oo(2,n)$ are easily checked using the matrix  expressions given in Section \ref{sec.algebraic-facts}.

   \begin{enumerate}[(i)]
   \item{For every $0 \neq u \in \lieg_0$, there exists $\zeta \in E_1^{\perp}=\lieg_{\beta- \alpha} \oplus \lieg_{- \alpha}$ such that $0 \neq [u, \zeta] \in \lieg_{-1}$.}
   \item{For every $0 \neq u \in \lieg_{+1}$, there exists $\zeta \in \lieg_{-\alpha}$ such that $0 \neq [\zeta,u] \in \lieg_0$.}
   \end{enumerate}

Given $\zeta \in E_1^{\perp}$, let $\hat{Z}$ be the associated $\omega$-constant vector field. Equation (\ref{eq.killing.cartan}) gives
\[ \hat{Z}.(\xi_0+\xi_{+1})(\hx)= [\xi_0(\hx)+\xi_{+1}(\hx),\zeta]-\kappa_{\hx}(E_1,\zeta) \qquad \forall \ \hx \in \hat{V} \]
 The property of the curvature function on $\hat{V}$ given by Lemma \ref{lem.curvature-values} (2) yields $\kappa_{\hx}(E_1,\zeta)=0$.
If $\xi_0(\hx) \not =0$ for some $\hx \in \hat{V}$, let $\zeta \in E_1^{\perp}$ such that $0 \neq [\xi_0(\hx), \zeta] \in \lieg_{-1}$. On the other hand $[\xi_0(\hx), \zeta]=\hat{Z}.(\xi_0+\xi_{+1})(\hx)-[\xi_{+1}(\hx),\zeta]$, but the right-hand side is in $\liep$, a contradiction.
Next suppose $\xi_{+1}(\hx) \not =0$ for some $\hx \in \hat{V}$, and take $\zeta \in \lieg_{-\alpha}$ such that $0 \neq [\xi_{+1}(\hx),\zeta] \in \lieg_0$. But then equality $[\xi_{+1}(\hx),\zeta]=(\hat{Z}.\xi_{+1})(\hx)$ yields a contradiction, since the right-hand term belongs to $\lieg_{+1}$.
\end{proof}

Next recall the following:
\begin{lemma}\cite[Lemma 2.5]{pecastaing.solvable}
\label{lem.vector-fields-proportional}
    Let $X$ and $Z$ be two conformal vector fields on a connected pseudo-Riemannian manifold $(M,g)$ of dimension $\geq 3$. If  $X_x$ and $Z_x$ are proportional for every $x \in M$, 
    then $X$ and $Z$ are collinear.
\end{lemma}

Recall from Section \ref{subsec.polarization} that on the universal cover $\tilde{M}$ there is a globally defined conformal vector field $\tilde{X}$ extending any local lift of $X$, and that any conformal deck transformation of $\tilde{M} \rightarrow M$ preserves the line field $\tilde{\mathcal{D}} = \BR \tilde{X}$.  Every conformal deck transformation pushes $\tilde{X}$ forward to another conformal vector field, so it acts on $\tilde{X}$ by a scalar, by Lemma \ref{lem.vector-fields-proportional}.  Lifting $\tilde{X}$ to the Cartan bundle of $\tilde{M}$ and evaluating with $\omega$ yields a map $\tilde{\eta}$ to ${\bf P}(\lieg)$ which descends to a well-defined, $P$-equivariant map $\eta : \hm \rightarrow {\bf P}(\lieg)$, such that $\left. \eta \right|_{\hat{V}} = [\left. \omega \right|_{\hat{V}}(\hat{X})] \equiv [E_1].$

Denote $o = [E_1] \in {\bf P}(\lieg)$, and by $\mathcal{O}$ the orbit $P.o$. 

\begin{proposition}
    \label{prop.mono-orbite}
    The image $\eta(\hm)$ equals the $P$-orbit $\mathcal{O}$.
\end{proposition}

\begin{proof}
Lemma \ref{lem.valeurs-X} ensures that $\eta$ takes values in the orbit $\mathcal{O}$ on $\hat{V}.P$, an open subset of $\hm$. Because $\eta$ is analytic, $\eta(\hm) \subseteq {\mathcal{O}}^Z$, the Zariski closure of $\mathcal{O}$ in ${\bf P}(\lieg)$. As $P$ is an algebraic subgroup of $\OO(2,n)$, and orbits of algebraic actions are locally closed, the orbit $\mathcal{O}$ is open in ${\mathcal{O}}^Z$. 

\begin{lemma}
    \label{lem.closure}
    Any point in $\overline{\mathcal{O}} \setminus \mathcal{O}$ belongs to $\bf P(\liep) \subset \bf P(\lieg)$.
\end{lemma}

\begin{proof}
  Let $x \in \overline{\mathcal{O}} \setminus \mathcal{O}$, with $(p_k) \subset P$ such that $p_k.o \to x$. Write $p_k=\exp(T_k)g_k$, with $T_k \in \lieg_{+1}$ and $g_k \in G_0$ for all $k$.  Recall that  $G_0$ is isomorphic to $\BR^* \times \OO(1,n-1)$. Use the Iwasawa decomposition to write $g_k=m_ka_kn_k$ where $m_k$ is in a maximal compact subgroup of $G_0$, $a_k \in A$, and $n_k \in U^+$. Observe  that $m_k$ normalizes $\lieg_{+1}$, hence replacing  $T_k$ by $\Ad(m_k)^{-1}(T_k)$, we may write $p_k=m_k\exp(T_k)a_kn_k$. 
Choose a Euclidean scalar product on $\lieg$ such that all the rootspaces are mutually orthogonal, and denote $| \cdot |$ the associated norm.  Decomposing $T_k$ along the rootspaces of $\lieg_{+1}=\lieg_{\alpha+\beta}\oplus \lieg_{\alpha} \oplus \lieg_{\alpha-\beta}$ gives 
$$T_k=\lambda_k u_{\alpha+\beta}(k)+ \mu_k u_{\alpha}(k)+ \nu_k u_{\alpha- \beta}(k)$$
with $|u_{\alpha+\beta}(k)|=|u_{\alpha}(k)|=|u_{\alpha- \beta}(k)|=1$
and $\lambda_k, \mu_k, \nu_k \geq 0$. After passing to a subsequence of $(p_k)$, we may assume that $m_k$ converges to $m \in P$, and $u_{\alpha+\beta}(k), u_{\alpha}(k)$, and $u_{\alpha- \beta}(k)$ converge to $u_{\alpha+\beta}, u_{\alpha}$, and $u_{\alpha- \beta}$, respectively. Replacing $x$ by $m^{-1}.x$ does not affect the conclusions of the lemma, so we will assume henceforth that $(m_k)$ is trivial. 
  Because the groups $A,U^+,$ and $\exp(\lieg_{\alpha+\beta})$ fix $o$, and because $\lieg_{+1}$ is abelian, we are left with:
  $$ p_k.o= \exp(T_k) a_k n_k.o = \exp(\mu_k u_{\alpha}(k)+ \nu_k u_{\alpha- \beta}(k)).o$$
  Consider  $v_k:=\Ad(\exp(\mu_k u_{\alpha}(k)+ \nu_k u_{\alpha- \beta}(k)))(E_1)$; it can be written:
  $$ v_k=E_1+\mu_k [u_{\alpha}(k),E_1]+\nu_k [u_{\alpha- \beta}(k),E_1]+w_k, \ \mbox{where } w_k \in \lieg_{+1}.$$
  
   Now $[u_{\alpha}(k),E_1] \rightarrow [u_{\alpha},E_1] \in \lieg_{\beta} \backslash \{ 0 \}$ and $[u_{\alpha- \beta}(k),E_1] \rightarrow [u_{\alpha- \beta},E_1] \in \liea \backslash \{ 0 \}$. 
 If $(\mu_k)$ and $(\nu_k)$ were both bounded, then a subsequence of $\exp(\mu_k u_{\alpha}(k)+ \nu_k u_{\alpha- \beta}(k))$ would have a limit $\ell$ in $P$, giving $x=\ell.o$. The assumption that $x$ is in $\overline{\mathcal{O}} \backslash \mathcal{O}$ thus forces $\max(\mu_k, \nu_k) \to + \infty$. 
 It follows that the component of $v_k$ on $\liea\oplus\lieg_{\beta}$ tends to infinity; in particular the component of $v_k$ on $\liep$ tends to infinity, while its component on $\lieg_{-1}$ is constant equal to $E_1$. Thus 
  \[ x=\lim_{k \to \infty}[\Ad(\exp(\mu_k u_{\alpha}(k)+ \nu_k u_{\alpha- \beta}(k)))(E_1)] \in \bf P(\liep) \]
  \end{proof}

A point $\hat{x} \in \hm$ where $\eta(\hat{x}) \in {\bf P}(\liep)$ projects to a point $x\in M$ at which a local projection  $X$ of $\tilde{X}$ vanishes.  But $\tilde{X}$ is never $0$ on $\tilde{M}$, therefore any $X$ is also nonvanishing.  Therefore $\eta(\hm)$ does not meet $\overline{\mathcal{O}} \backslash \mathcal{O}$, so $\eta(\hm) = \mathcal{O}$, as claimed.
\end{proof}

Because $Q$ is the stabilizer in $P$ of $o$, the inverse image
$\widehat{R}:=\eta^{-1}( o)$ is the analytic reduction of $\hm$ to $Q$ claimed in point (3) (c) of Theorem \ref{thm.pp-wave}. 
Observe that $\eta$ was defined in terms of the lightlike line field $\mathcal{D}$ and its local sections $X$ corresponding to conformal vector fields. The $\Conf^{loc}(M)$-invariance of $\widehat{R}$ then follows from the $\Conf^{loc}(M)$-invariance of $\mathcal{D}$ and the invariance up to scale of these local sections.

The next step is to propagate the local gravitational pp-wave metrics over all of $M$, which will be obtained as reductions of $\widehat{R}$. Let $\widehat{R}_V = \hat{V}.Q$, an open subset of $\widehat{R}$. Let $\hat{x} \cdot q$ be an arbitrary point of $\widehat{R}_V$, where $\hat{x} \in \hat{V}$ and $q \in Q$.  The evaluation $\omega(T_{\hat{x} \cdot q} \widehat{R}_V)$ contains $\Ad(q)(\lieg_{-1} + \lieu^+)$ and $\lieq$, which are easily seen to both be contained in $\lieg_{-1} \oplus \lieq$.  After comparing dimensions, we conclude $\omega(T_{\hat{y}} \widehat{R}_V) = \lieg_{-1} + \lieq$ for all $\hat{y} \in \widehat{R}_V$. By analyticity, this property holds on all of $\widehat{R}$, that is, $\widehat{R}$ is parallel with respect to the subalgebra $\lieg_{-1} + \lieq$.

The line $\BR E_1$ is $\Ad(Q)$-invariant, therefore  $\left. \omega \right|_{\widehat{R}}^{-1}(\BR E_1)$ is a right-$Q$-invariant distribution on $\widehat{R}$, descending to $\mathcal{D}$ on $M$. The distribution $\omega^{-1}(\lieg_{-1} \oplus \lieu^{+})$ on $\widehat{R}$ is right-$Q$-invariant because $\Ad(Q)$ leaves $\lieg_{-1} + \lieu^+$ invariant. The right translates by elements of $Q$ of the integral leaf $\hat{V}$ are  additional integral leaves foliating $\widehat{R}_V$. Finally, the integrability extends from the open subset $\widehat{R}_V$ to all of $\widehat{R}$ by analyticity. Lemma \ref{lem.curvature-values} gives that $\operatorname{Im} \kappa_{\hx} \subset \lieu^+$ and $\kappa_{\hx}(\xi,\eta)=0$ for every $\hx \in \hat{V}$ and every $\xi,\eta$ in $E_1^{\perp}$. Because $\lieu^+$ is $\Ad(Q)$-invariant, and $E_1^{\perp}$ is $\Ad(Q)$-invariant modulo $\liep$, these properties both persist for $\hat{x} \in \widehat{R}_V$, and they extend to all of $\widehat{R}$ by analyticity.  

For any $x \in M$, let $\hat{x} \in \widehat{R}_x$.  Let $\hat{W} \subset \widehat{R}$ be an integral leaf of $\omega^{-1}(\lieg_{-1} \oplus \lieu^+)$ through $\hx$. The second point of Proposition \ref{prop.brinkmann} gives a gravitational $pp$-wave on a neighborhood $W$ of $x$ with parallel lightlike vector field $X$ on $W$ equal the projection of $\left. \omega \right|_{\hat{W}}^{-1}(\hat{E_1})$. It is a section of $\mathcal{D}$ and a local conformal vector field.  The proof of point $(3) (c)$ in Theorem \ref{thm.pp-wave} is complete.

It remains only to compute the bracket relation between $Y$ and any local vector field $X$ as in (c).

\begin{lemma}
Let $V$ be the neighborhood of $x_0$ given by Theorem \ref{thm.pp-wave} (2), and let $X \in \mathcal{X}(V)$ be the parallel lightlike vector field for the gravitational pp-wave metric $g_0$ on $V.$
The vector fields $Y$ and $X$ on $V$ satisfy the bracket relation  $[Y,X] = 2 X$. 
\end{lemma}

\begin{proof}
From the linearization and the construction of $X$, the lightlike tangent vector $X_{x_0}$ is the fastest eigenvector of $D_{x_0}  \varphi^t_Y$, with eigenvalues $ e^{-2t}$. The flow of $Y$ preserves the lightlike line distribution $\mathcal{D}$, so $[Y,X]$ is collinear to $X$ at every point of $V$.  By Lemma \ref{lem.vector-fields-proportional}, there is a constant $\lambda$ such that $[Y,X]=  \lambda X$.  At $x_0$, we have 
$$[Y,X]_{x_0} = \left. \frac{\mathrm{d}}{\mathrm{d} t} \right |_{t=0} (D_{x_0} \varphi^t_Y)^{-1}X_{x_0} = 2X_{x_0}$$
Therefore $\lambda = 2$.
\end{proof}

To complete the proof of Theorem \ref{thm.pp-wave} part (3) (d), take the unique extension $\tilde{X}$ of any lift of $X$ to the universal cover $\tilde{M}$.  The lift $\tilde{Y}$ of $Y$ to $\tilde{M}$ satisfies $[\tilde{Y},\tilde{X}] = 2 \tilde{X}$ on a nonempty open subset of $\tilde{M}$ by the above lemma.  Therefore by analyticity, the relation holds on all of $\tilde{M}$.  Then for any local projection $X$ of $\tilde{X}$ to an open subset of $M$, the bracket is $[Y,X] = 2X.$

(Note that the rescaling of $Y$ mentioned in the statement of Theorem \ref{thm.pp-wave} 3(d) happened at the beginning of this Section \ref{sec.pp-wave}, where we arranged that $a=b=1$.)

\section{Linear balanced singularity implies conformal flatness} 
\label{sec.balanced}

This section will conclude the proof of Theorem \ref{thm.local.lichnerowicz}. The strategy, as indicated in Section \ref{sec.outline}, is to follow the orbit of a point near the initial balanced singularity to an $\alpha$-limit point and to study the flow at this new limit point.  It can be described relatively precisely thanks to the Cartan connection.  The conformal vector field at this point gives rise, with the help of a recurrence argument, to a polarization, which is shown to be distinct from the one found in Section \ref{sec.polarization} above.  Then Lemma \ref{lem.polarization} implies conformal flatness.  The results of this section are summarized in the following proposition:

\begin{proposition}
\label{prop.balanced}
 Let $(M,g)$ be a closed, real-analytic, Lorentzian manifold. Let $Y$ be a conformal vector field on $M$. If $Y$ admits a singularity $x_0$ which is balanced, then $(M,g)$ is conformally flat.    
\end{proposition}

The proposition is proved by contradiction: we assume that $(M,g)$ is not conformally flat, and apply Theorem \ref{thm.pp-wave}. As in point (3) of that theorem, let $\mathcal{D}$ be the lightlike line field with respect to which $(M,g)$ is polarized, and set
$$\mathcal{C} = \{x \in M \ | \ Y_x \in \mathcal{D}_x\}$$ 
Note that $\mathcal{C}$ is a closed, even real-analytic.  It is also $\Kill^{loc}_Y$-invariant, since ${\mathcal D}$ is $\Conf^{loc}(M,g)$-invariant.

Let $X$ be a local conformal vector field in a neighborhood of $x_0$ with values in $\mathcal{D}$, as in Theorem \ref{thm.pp-wave} (3c). The bracket relation in $[Y,X]=2X$ of (3d) gives 
$$D_{x_0} \varphi_X^t (Y_{x_0}) = Y_{\varphi_X^t(x_0)} + 2tX_{\varphi_X^t(x_0)}$$ 
for all sufficiently small $t$. Choose $x_1:=\varphi_X^{t_1}(x_0)$ for such small ${t_1} >0$.  Then $x_1 \in \mathcal{C}$, and $Y_{x_1} \not =0$ because $Y_{x_0} = 0$ but $X$ is nonvanishing.  The \emph{$\alpha$-limit set} of $x_1$ for the flow $\{ \varphi^t_Y\}$, that is, the set of accumulation points of sequences $(\varphi^{t_k}_Y(x_1))$, for $t_k \to - \infty$, is denoted $\alpha(x_1)$.  A key fact in the proof will be:

\begin{lemma}
\label{lem.nosingularity}
   The vector field $Y$ does not vanish at any point in $\alpha(x_1).$
   \end{lemma}
 \begin{proof}
   Once again, we assume the contrary, namely, that $Y_y = 0$ for $y \in \alpha(x_1)$.  
   There is a decreasing sequence $t_k \to - \infty$ such that $\varphi^{t_k}_Y(x_1) \to y$.
Let $X$ be the nonvanishing lightlike conformal vector field belonging to $\mathcal{D}$, given by Theorem \ref{thm.pp-wave} (3)(c), defined on a neighborhood of $y$.  This neighborhood, call it $U$, can be chosen to be a flow-box for the local flow of $X$, such that there is a transversal $T \subset U$ containing $y$ and $\epsilon > 0$ with
$$\Phi : (s,z) \in (-\epsilon,\epsilon) \times T \mapsto \varphi^s_X(z)$$ 
a diffeomorphism onto its image.  Lastly, since $x_1 \neq y$, it can be assumed that $x_1 \not \in U$. 

 Write $\varphi^{t_k}_Y(x_1)=\Phi(s_k,y_k)$ with $y_k \to y$ in $T$.  Let $\gamma(s):=\Phi(s,y) = \varphi^t_X(y)$, parametrizing a null geodesic segment through $y$, for $s \in (-\epsilon, \epsilon)$.  The relation $[Y,X]=2X$ of Theorem \ref{thm.pp-wave} (3)(d) gives $\varphi^t_Y(\gamma(s)) = \gamma(e^{-2t}s)$ for $t \geq 0$. The backwards $\{ \varphi^t_Y \}$-semi-orbit of $x_1$ does not meet the segment $\gamma$: if $\varphi^{t}_Y(x_1)=\gamma(s)$ for some $t \leq 0$ and some $s \in (-\epsilon,\epsilon)$, then $x_1=\gamma(e^{2 t}s)$ would belong to $U$, contradicting the choice of $U$. In particular, $y_k \not = y$ for all $k$. 

 On the other hand, the segments $\Delta_k=\{ \Phi(s,y_k)  \  | \ s \in (-\epsilon,\epsilon)\}$ are completely contained in this semi-orbit $\{\varphi^t_Y(x_1)\}_{t \leq 0}$. Each such semgent intersects the semi-orbit in a point $p_k = \Phi(s_k,y_k)$, and they are tangent at this point because $p_k = \varphi^{t_k}_Y(x_1) \in \mathcal{C}$.  If a point of $\Delta_k$ were outside the semi-orbit, it would mean that for some $T > 0$, the set 
 $$\{\varphi^t_Y(x_1)\}_{t \leq -T} \subset \Delta_{y_k}$$
 But that would contradict $y \in \alpha(x_1)$. 
 
 Now for each $k \in \BN$, let $f_k: (-\epsilon,\epsilon) \to \BR^*$ be such that
 $$Y_{\Phi(s,y_k)} =f_k(s) X_{\Phi(s,y_k)} \ \forall s \in (-\epsilon,\epsilon)$$
 And let $f : (-\epsilon,\epsilon) \to \BR$ be such that 
 $$Y_{\gamma(s)}=f(s) X_{\gamma(s)} \ \forall s \in (-\epsilon,\epsilon)$$
  By the relation $\varphi^t_Y(\gamma(s)) = \gamma(e^{-2t}s)$, 
the only $0$ of $f$ is $s=0$ and $f$ takes opposite signs on $(-\epsilon, 0)$ and $(0,\epsilon)$.  
For any $0 \neq \sigma \in (-\epsilon,\epsilon)$, the values $f_k(\sigma) \to f(\sigma)$ and $f_k(-\sigma) \to f(-\sigma)$. However $f(\sigma)f(-\sigma)<0$ while $f_k(\sigma)f_k(-\sigma)>0$ for all $k$, a contradiction.
\end{proof}

Because the set $\alpha(x_1)$ is compact and $\{ \varphi^t_Y\}$-invariant, it contains recurrent points for $\{ \varphi^t_Y\}$.  

\begin{proposition}(compare \cite[Prop. 5.1]{frances.open.dense},
  \cite[Prop. 3.5]{frances.lorentz.3d})
  \label{prop.recurrence}
  Let $(M,g)$ be a compact, real-analytic
  pseudo-Riemannian manifold, and let $Y$ be a conformal vector field such that the flow $\{ \varphi^t_Y \}$ has noncompact closure in $\Conf(M,g)$. 
 At each recurrent 
  point $z$ for $\{ \varphi^t_Y \}$, there exists a  local conformal vector field $Z$ commuting with $Y$, vanishing at $z$, and admitting a noncompact isotropy flow at $z$. 
    \end{proposition}
  See Section \ref{sec.stratification} to recall the definition of the isotropy on the Cartan bundle. 

Now let $\widehat{R} \subset \hm$ be the $\Conf^{loc}$-invariant reduction given by Theorem \ref{thm.pp-wave} (3)(c), 
 to the group $Q<P$ stabilizing the subspace $\BR E_1 = \lieg_{\beta-\alpha}$ under the adjoint representation of $P$ on $\lieg$. 
 Recall from the construction of $\widehat{R}$ that for a neighborhood $U$ of $x_0$ in which $X$ is defined,  
 the set $\widehat{R}_U$
 comprises points $\hx$ where $\omega_{\hx}(\hat{X}) \in \BR E_1$ modulo $\liep$, where $\hat{X}$ is the lift of $X$ to $\hm$.
Recall that for suitable $\hat{x}_0 \in \hm_{x_0}$ 
\begin{equation*}
\omega_{\hx_0}(\hat{Y}) =
\begin{pmatrix}
1 & & & & \\
& -1 & & & \\
& & A & & \\
& & & 1 & \\
& & & & -1
\end{pmatrix}
=: Y_0
\label{eq.holonomyYatx0}
\end{equation*}

with $A \in \so(n-2)$. Recall the decomposition of $\so(2,n)$ from Section \ref{sec.algebraic-facts}.  Denote by $Y_{\liea}$ and $Y_{\liem}$ the $\liea$- and $\liem$-components of $Y_0$.

\begin{lemma}
\label{lem.values-on-alpha-limit}
    For any $z \in \alpha(x_1)$, there exists $\hat{z} \in \widehat{R}_z$ such that $\omega(\hat{Y}_{\hat{z}})=Y_{\beta-\alpha}'+Y_{\liea} + Y_{\liem}'$, where $Y_{\beta-\alpha}' \in \lieg_{\beta-\alpha} \backslash \{0\}$ and $Y_{\liem}' \in \liem$.
\end{lemma}

\begin{proof}
Recall that $K$ is semibasic, so $K_{\hx_0}(\hat{Y},\hat{X}) = 0$. By Lemma 2.1 of \cite{bfm.zimemb}, this yields
\[2\omega_{\hx_0}(\hat{X}) = \omega_{\hx_0}([\hat{Y},\hat{X}]) = -[ \omega_{\hx_0}(\hat{Y}),\omega_{\hx_0}(\hat{X})] = - [Y_0,\omega_{\hx_0}(\hat{X})].\]
Hence, $\ad(Y_0)\,\omega_{\hx_0}(\hat{X}) = -2\,\omega_{\hx_0}(\hat{X})$. 
Considering the action of $\ad(Y_0)$ on the root spaces of $\lieo(2,n)$, this implies that $\omega_{\hx_0}(\hat{X}) \in \lieg_{\beta-\alpha}$. Therefore $\hx_0 \in \widehat{R}$.

Let $\hx_1 := \varphi_{\hat{X}}^{t_1}(\hx_0)$.  From the relation 
$$(\varphi_{\hat{X}}^t)_* \hat{Y}_{\hx_0} = \hat{Y}_{\varphi_{\hat{X}}^t(\hx_0)} + 2t \hat{X}_{\varphi_{\hat{X}}^t(\hx_0)}$$
it follows htat $\hx_1 \in \widehat{R}$ and $\omega_{\hx_1}(\hat{Y}) = Y_0 - 2t_1 \omega_{\hx_0}(\hat{X})$. 
Set $Y_{\beta-\alpha} := -2t_1 \omega_{\hx_0}(\hat{X})$, so that $\omega_{\hx_1}(\hat{Y}) = Y_{\beta-\alpha} + Y_0$; note that $Y_{\beta-\alpha} \neq 0$.

Let $t_k \rightarrow -\infty$ such that $\varphi^{t_k}_Y(x_1)=z$. Lifting to $\widehat{R}$, there is a sequence $(q_k)$ in $Q$ such that $\varphi_Y^{t_k}(\hx_1).q_k \to \hat{z}$, for some $\hat{z} \in \widehat{R}_{z}$. This the value $\omega_{\hz}(\hat{Y})$ is in the closure of $\Ad(Q)(Y_{\beta-\alpha} + Y_0)$ in $\lieg$. Decompose $Q$ as a semi-direct product $(A . M) \ltimes ( G_{\beta} . G_{\alpha+\beta})$, and write $q = q_0 e^{X_{\beta}} e^{X_{\alpha+\beta}}$ for the corresponding decomposition of an element $q \in Q$.  Now compute
\begin{align*}
    \Ad(q) (Y_{\beta-\alpha}+Y_0) & = \lambda Y_{\beta-\alpha} + \Ad(q_0) (Y_0 + [X_{\beta},Y_0]) \\
    & = \lambda Y_{\beta-\alpha} + Y_{\liea} + \Ad(q_0) Y_{\liem} + \underbrace{\Ad(q_0) [X_{\beta},Y_0]}_{\in \lieg_{\beta}}
\end{align*}
for some $\lambda \in \BR^*$, since $Y_0$ centralizes $\lieg_{\alpha+\beta}$. 
Consequently, any element $T$ in the closure of $\Ad(Q)(Y_{\beta-\alpha}+Y_0)$ will be of the form 
\begin{align*}
    T = T_{\beta-\alpha} + Y_{\liea} + T_{\liem} + T_{\beta}. 
\end{align*}
Since $[Y_{\liea},T_{\beta}] = - T_{\beta}$ and $T_{\beta}$ centralizes the other terms, this gives 
\begin{align*}
    \Ad(e^{-T_{\beta}}) T = T_{\beta-\alpha} + Y_{\liea} + T_{\liem}.
\end{align*}
Hence, up to right-multiplication of $\hz$ by some $q \in G_{\beta} < Q$, the value $\omega_{\hz}(\hat{Y})=Y_{\beta-\alpha}' + Y_{\liea} + Y_{\liem}'$, with $Y_{\beta-\alpha}' \neq 0$ since $Y_z \neq 0$ from Lemma \ref{lem.nosingularity}.
\end{proof}
 
 Now let $\hz \in \widehat{R}_z$ as given by Lemma \ref{lem.values-on-alpha-limit}. Because $Z_z =0$ and $Z$ preserves $\widehat{R}$, the value $\omega(\hat{Z}_{\hz}) \in \lieq$. Examining the curvature form at $\hz$ leads to $[\omega(\hat{Y}_{\hz}),\omega(\hat{Z}_{\hz})]=0$. The possible values for $\omega(\hat{Z}_{\hz})$ are then determined by the following algebraic result:
   
\begin{lemma}
\label{lem.centralizer-of-omega(Y)}
   Any $S \in \lieq$ centralizing an element of the form $T = T_{\beta-\alpha} + Y_{\liea} + T_{\liem}$, with $Y_{\liea}$ as above and $T_{\beta-\alpha} \not =0$ in $\lieg_{\beta-\alpha}$, is in the $\Ad(Q)$-orbit of an element of the form either 
    \begin{equation*}
         \diag(\lambda,\lambda,B,- \lambda, - \lambda) , \ \lambda \neq 0 \text{ or } \diag(0,0,B,0,0) + S_{\alpha+\beta}.
     \end{equation*}    
\end{lemma}

\begin{proof}
Decompose $S$ along the root spaces of $\lieq$, writing $S=S_{\liea}+S_{\liem}+S_{\beta}+S_{\alpha+\beta}$. As $[T,S]=0$, 
$$[T,S]_{\beta-\alpha}=[T_{\beta-\alpha},S_{\liea}]=0=(\alpha(S_{\liea})-\beta(S_{\liea}))T_{\beta-\alpha}$$
Since $T_{\beta-\alpha} \not = 0$, necessarily $\alpha(S_{\liea})=\beta(S_{\liea})$.

Similar calculation of the component along $\lieg_{\beta}$ of $[T,S]$ yields $0 = -S_{\beta}+[T_{\liem},S_{\beta}]$. Since $\ad(T_{\liem})|_{\lieg_{\beta}} : \lieg_{\beta} \to \lieg_{\beta}$ is skew-symmetric with respect to the positive definite bilinear form $B_{\theta} = - B(\theta \cdot, \cdot)$, its eigenvalues are purely imaginary. Therefore, $S_{\beta}=0$.

Now $S$ is of the form $S=S_{\liea}+S_{\liem}+S_{\alpha+\beta}$, with $\alpha(S_{\liea})=\beta(S_{\liea})$. If $\alpha(S_{\liea}) \not = 0$, conjugation by a suitable element of $G_{\alpha+\beta} < Q$ eliminates the component $S_{\alpha+\beta}$. The desired conlcusion now follows, with $\lambda = \alpha(S_{\liea})$.
\end{proof}

Now apply Lemma \ref{lem.centralizer-of-omega(Y)} to $S=\omega(\hat{Z}_{\hz})$ and $T = \omega(\hat{Y}_{\hz})$. 

If $\omega(\hat{Z}_{\hz})$ is of the form $\diag(\lambda,\lambda,B,- \lambda, - \lambda) $, then $Z$ has linear balanced dynamics near $z$, yielding  a polarization $\mathcal{D}'$ of the conformal structure on $M$ by point 3(b) of Theorem \ref{thm.pp-wave}. On the other hand, the singular locus of $Z$ coincides locally with the $\{ \varphi^t_Y\}$-orbit of $z$. Because $z \in {\mathcal C}$ and $Y_z \not =0$,  the tangent to this orbit at $z$ is ${\mathcal D}_z$. Point 3(a) of Theorem \ref{thm.pp-wave} then ensures that $\mathcal{D} \neq \mathcal{D}'$, and Lemma \ref{lem.polarization} implies that $M$ is conformally flat, a contradiction.
 
Next suppose $\omega(\hat{Z}_{\hz}) = \diag(0,0,B,0,0) + Z_{\alpha+\beta}$. Let $\{ h^t\}< P$ be the one-parameter group of generated by this element. By Proposition \ref{prop.recurrence}, it is noncompact, hence $Z_{\alpha+\beta} \not =0$. Recall that $\{ \varphi^t_Z \}$ is linearizable around  $z$ if and only if $\{ h^t\}$ is conjugate in $P$ into $G_0$, or equivalaently, if and only if $\{h ^t \}$ is linearizable on $\Ein$around $[e_0]$.  Observe that $\{ D_{[e_0]}h^t\}=\{e^{tB}\}$, a compact subgroup of $\operatorname{SO}(1,n-1)$. Hence, if $\{h^t\}$ were linearizable around $[e_0]$, the $\{ h^t\}$-orbit of every point $x \not = [e_0]$ close to $[e_0]$ would be contained in a compact subset not containing $[e_0]$. However, the action of $\{h^t\}$ on the photon $\Delta:=[\Span(e_0,e_n)]$ is by a nontrivial projective parabolic flow. For such a flow, there is $x \not = [e_0]$ in $\Delta$, arbitrarily close to $[e_0]$, such that $\lim_{t \to + \infty}\varphi^t_Z(x) =  [e_0]$: contradiction. This completes the proof of Proposition \ref{prop.balanced} and also of Theorem \ref{thm.local.lichnerowicz}.
 
\bibliographystyle{amsalpha} 
\bibliography{karinsrefs}

\end{document}